\documentclass[11pt,a4paper,halfparskip,twoside]{article}
\usepackage[english]{babel}
\usepackage{amsmath,amssymb,amstext}
\usepackage{psfrag,mathrsfs}

\usepackage{graphicx,color}

\newtheorem{theorem}{Theorem}[section]
\newtheorem{lemma}[theorem]{Lemma}
\newtheorem{proposition}[theorem]{Proposition}

\newtheorem{definition}[theorem]{Definition}
\newtheorem{remark}[theorem]{Remark}

\newenvironment{proof}[1][Proof:]{\begin{trivlist}
\item[\hskip \labelsep {\bfseries #1}]}{\end{trivlist}}

\newcommand{\qed}{\nobreak \ifvmode \relax \else
      \ifdim\lastskip<1.5em \hskip-\lastskip
      \hskip1.5em plus0em minus0.5em \fi \nobreak
      \vrule height0.75em width0.5em depth0.25em\fi}

\newcommand{\R}{\mathbb{R}}
\newcommand{\N}{\mathbb{N}}

\usepackage[automark]{scrpage2}
\usepackage{bibgerm}
\begin{document}

\thispagestyle{empty}


\begin{center}

\LARGE\textbf{Optimal Reliability in Design for Fatigue Life}\\\vspace{0.5cm}


\Large\textbf{Part I -- Existence of Optimal Shapes}\\

\vspace{0.9cm} \noindent{Hanno Gottschalk\footnote{Bergische
Universit\"at Wuppertal, Fachbereich Mathematik und
Naturwissenschaften, Email: hanno.gottschalk@uni-wuppertal.de} and Sebastian
Schmitz\footnote{Universita della Svizzera Italiana, Istituto di
Scienze Computazionali,  Email: sebastian.schmitz@usi.ch
and
Siemens AG Energy, M\"ulheim an
der Ruhr, Germany, Email: schmitz.sebastian@siemens.com}}\\
\end{center}

\noindent{\small{\bf Abstract:} The failure of a component often
is the result of a degradation process that originates with the
formation of a crack. Fatigue describes the crack formation in the
material under cyclic loading. Activation and deactivation
operations of technical units are important examples in
engineering where fatigue and especially low-cycle fatigue (LCF)
play an essential role. A significant scatter in fatigue life for
many materials results in the necessity of advanced probabilistic
models for fatigue. Moreover, optimization of reliability is of
vital interest in engineering, where with respect to fatigue the
cost functionals are motivated by the predicted probability for
the integrity of the component after a certain number of load
cycles. The natural mathematical language to model failure, here
understood as crack initiation, is the language of spatio-temporal
point processes and their first failure times.
The local crack formation intensities thereby need to be modeled
as a function of local stress states and thus as a function of the
derivatives of the displacement field $u$ obtained as the solution
to the PDE of linear elasticity. This
translates the problem of optimal reliability in the framework of
shape optimization.
The cost functionals derived in this way for realistic optimal
reliability problems are too singular to be $H^1$-lower
semi-continuous as many damage mechanisms, like LCF, lead to crack
initiation as a function of the stress at the component's surface.
Realistic crack formation models therefore impose a new challenge
to the theory of shape optimization. In this work, we have to
modify the existence proof of optimal shapes, for the
case of sufficiently smooth shapes using elliptic regularity,
uniform Schauder estimates and compactness of strong solutions via
the Arzela-Ascoli theorem. This result applies to a variety of
crack initiation models and in particular applies to a recent
probabilistic model for LCF.}

\vspace{.2cm}


\section{Introduction}
A design being made from material will fail if the material
degradation due to loading exceeds certain limits. Reliability,
i.e. the absence of failure, is thus the ultimate goal of
structural design.  Whether a design will operate safely under
certain load conditions depends on the failure mechanisms which
are very diverse for different material classes and operation
conditions. Degradation can occur as a function of operating time,
which is e.g. the case for creep damage. Or it can occur when
small plastic deformations under cyclic loading pile up and result
in a crack. This is called fatigue which can be differentiated
into high-cycle fatigue (HCF) and low-cycle fatigue (LCF)
\cite{Harders_Roesler,Vormwald}.

A common feature of crack initiation over diverse classes of
material and damage mechanisms is its probabilistic nature, see
e.g.\ \cite{Vormwald}. A natural mathematical language to capture
the random times and locations of crack initiation is the language
of point processes \cite{Kallenberg}. Associated to a
spatio-temporal point process is a failure time, which is well
known from the theory of renewable models in reliability
statistics \cite{Escobar_Meeker}.  The probability that the time
to failure -- here understood as the time that passes until the
formation of the first crack -- will be larger than a certain
warranty time or service interval length $t^*$ is an important
quantity in engineering. The maximization of this quantity sets
the problem of optimal reliability.

A central part of engineering is about the choice of the form
$\Omega\subseteq\R^3$ of the component. Given material, surface
and volume loads that act on the component, it is desirable to
optimize failure and survival probabilities for the component at
$t^*$ as a function of $\Omega$. This embeds the quest of optimal
reliability into the established field of shape optimization
\cite{ABFJ,Arora,BS,Bucur_Buttazzo,DZ,Haug_Arora,Epp,Shape_Optimization_Haslinger,Sokolowski_Zolesio}
or synonymously PDE constrained optimization with shape control.
The natural choice for the  PDE is the PDE of (linear and
elliptic) elasticity
\cite{Braess,Elasticity_Ciarlet_1,Evans,Finite_Elements_Ern}.

In \cite{ABFJ,AJ,AK1,DZ}, the authors show generalized optimal
shapes with elasticity state equation by the homogenization
method, which translates shape optimization to a sizing problem.
In particular this allows compliance optimization, which otherwise
leads to an ill-posed problem. The usage of generalized shapes in
design however poses some additional problems. The existence
results in \cite{Shape_Optimization_Haslinger} and references
therein are based on shapes that fulfil some compactness
properties and cost functionals with lower $H^1$-semicontinuity.
Fujii \cite{Fujii} gives an interesting method to prove lower
semicontinuity for convex cost functionals given by volume
integrals of derivatives of the solution $u$ with respect to weak
$H^1$-topologies for scalar elliptic PDE. Although we essentially
follow  the book by M\"akinen and Haslinger, we are able do deal
with cost functionals that live on the boundary and depend on
stress or even higher derivatives on the bulk or boundary, which
is not covered in \cite{Shape_Optimization_Haslinger}. Also, we do
not require any convexity assumption. Further results on optimal
shapes are e.g.\ \cite{EU} dealing with the case of quasi-linear
cost functionals and \cite{LNT} for quadratic cost functionals not
depending on derivatives of the solution. For results of
eigenvalue optimization see \cite{DZ} and references given there.

In particular, cost functionals motivated by reasonable point
process models of the material failure mechanisms however fall
into this singular class. This is due to the fact that many
failure mechanisms are stress driven and generate surface cracks,
as it is the case for LCF. The associated crack intensities thus
depend locally on $\nabla u$ restricted to $\partial\Omega$, which
is an operation that is not well defined for the weak
$H^1$-solutions $u$ to the elasticity PDE. As a consequence, the
theory of weak solutions as in \cite{Shape_Optimization_Haslinger}
is insufficient for our purposes and the shape optimization
formalism needs to be extended to strong solutions with the help
of elliptic regularity theory
\cite{Agmon_Douglis_Nirenberg_1,Agmon_Douglis_Nirenberg_2,Evans,Elliptic_PDE_Gilbarg}.
This makes it necessary to revise existing proofs for the
existence of optimal shapes with a PDE constraint given by linear
elasticity. In particular, stronger $C^4$-boundary regularity is
considered that is not needed in the articles cited above.

Note that in the first sections, we treat the stochastic subject
of optimal reliability which leads to cost functionals that are
given by volume and surface integrals. In Section
\ref{section_SO_admissible_domains} to
\ref{section_optimal_reliability}, we focus on a class of shape
optimization problems which include the objective of optimal
reliability under linear elastic PDE constraints. Readers only
interested in the general theory of shape optimization can also
read these sections independently.
 The main result regarding the existence of
optimal shapes is Theorem \ref{theorem_existence_optimal_shape}.
The paper is organized as follows:

In Section \ref{section_stochastics_poisson_distribution}, we
review the theory of point processes essentially following
\cite{Kallenberg} with applications to reliability statistics
\cite{Escobar_Meeker} and formulate the optimal reliability
problem. Here, the notions of crack initiation models and their
associated failure times and probabilities are crucial. In
particular, if interaction between cracks is neglected, the
Poisson point process (PPP) plays a central role.  Here, we
connect the topic of optimal reliability to shape optimization
with the elasticity PDE as state equation and classify PPP models
according to their singularity. In Section \ref{section_prob_lcf},
we introduce a recently proposed and experimentally validated
crack initiation model for LCF in polycrystalline metal
\cite{SSBKRG}. This model has been applied to gas turbine
compressor design \cite{SKRG1} and extended to themomechanical
design of cooled gas turbine blades \cite{SKRG2}. Here, we give a
detailed account of its mathematical properties and show that it
fits to the framework of Section
\ref{section_stochastics_poisson_distribution}.

Section \ref{section_SO_admissible_domains} sets the stage for
shape optimization essentially following
\cite{Shape_Optimization_Haslinger} with an emphasis on uniform
$C^k$-regularity of shapes in the local epigraph parametrization
of shapes
\cite{DZ,Shape_Optimization_Haslinger,Sokolowski_Zolesio} with
$k=4$.  The following Section \ref{section_schauder_estimates}
treats existence and regularity of strong solutions to the
elasticity PDE based on regularity theory for elliptic systems
\cite{Agmon_Douglis_Nirenberg_1,Agmon_Douglis_Nirenberg_2}. In
order to check uniform regularity we revisit several results in
this field and check uniformity of estimates as a function of our
parametrization of shapes. In conclusion we are able to show that
the solutions to the elasticity PDE on admissible domains $\Omega$
fulfill uniform $[C^{3,\phi}(\Omega)]^3$-bounds where this space
stands for $3$-times differentiable functions with $\phi$-H\"older
continous third derivarive, $\phi\in (0,1)$. We also show that
such functions can be extended to $[C^{3,\phi}( \Omega^{\rm
ext})]^3$ with $\Omega^{\rm ext}\subseteq \R^3$ being some bounded
Lipschitz domain that contains all admissible shapes $\Omega$.
This section contains the technical core of this paper.

In the final Section \ref{section_optimal_reliability}, we derive
our main result regarding the existence of optimal shapes under
linear elastic PDE constraints (Theorem
\ref{theorem_existence_optimal_shape}). Here, we apply the results
of Section \ref{section_schauder_estimates} to the framework of
shape optimization developed in Section
\ref{section_SO_admissible_domains}. By the Arzela-Ascoli theorem,
balls in $[C^{3,\phi}(\Omega)]^3$ are compact. It is therefore
easy to show that the graph ${\cal G}$ that consists of pairs
$(\Omega,u(\Omega))$ of admissible shapes -- endowed with a
suitable $C^4$-topology -- and the associated solution to the
elasticity state equation with $C^{3,\phi}$-topology  possesses
the compactness of ${\cal G}$ required in the general shape
optimization formalism \cite{Shape_Optimization_Haslinger}. For a
minimizing sequence $(\Omega_n)_{n\in\mathbb{N}}$ of admissible
shapes with respect to a cost functional $J(\Omega,u(\Omega))$,
the implied $[C^{3,\phi}_0(\Omega^{\rm ext})]^3$-convergence of a
subsequence of extended solutions $(u^{\rm
ext}(\Omega_{n_k}))_{n_k\in\mathbb{N}}$ makes it easy to check the
(lower semi-) continuity properties for a large class of rather
singular cost functionals $J(\Omega,u)$. This concludes the proof
of existence of optimal admissible shapes. In particular, our
proof covers a large class of cost functionals that include
optimal reliability problems introduced in Sections
\ref{section_stochastics_poisson_distribution} and the
probabilistic LCF model from Section \ref{section_prob_lcf}, in
particular.

\section{A Mathematical Setting for Optimal Reliability}\label{section_stochastics_poisson_distribution}

Let us consider a bounded, open domain $\Omega\subset \R^3$ with
Lipschitz boundary $\partial\Omega$. Let $\Omega$ be interpreted
as the portion of the space filled with matter, i.e.\ the shape of
the component. Let $\mathscr{T}=\N_{0}$ or $=\bar \R_+$ be the
time axis, where time is either measured in discrete units (load
cycles) of in natural time. Let $dt$ denote either the continuous
or discrete Lebesgue measure on $\mathscr{T}$.

We define $\mathscr{C}=\mathscr{T}\times \bar\Omega$ as the
configuration space of crack initiations at time $t\in\mathscr{T}$
and at location $x\in\bar\Omega$ and endow $\mathscr{C}$ with the
standard metric topology.

Let $\mathscr{R}=\mathscr{R}(\mathscr{C})$ be the space of Radon
measures on $\mathscr{C}$, i.e.\ the set of measures $\gamma$ on
the measurable space $(\mathscr{C},\mathscr{B}(\mathscr{C}))$ such
that $\gamma(A)<\infty$ for $A\in \mathscr{B}(\mathscr{C})$
bounded. $\mathscr{B}(\mathscr{C})$ denotes the Borel sigma
algebra of the topological space $\mathscr{C}$. By $\mathscr{R}_c$
we denote the counting measures, i.e. the Radon measures
$\gamma\in\mathscr{R}$ such that $\gamma(B)\in\N_0$ for bounded,
measurable $B\subseteq \mathscr{C}$. In the given context
$\gamma\in\mathscr{R}_c$ encodes one particular history of
(multiple) crack initiations on the component $\bar\Omega$.  Given
such a history and some $B\in\mathscr{B}(\mathscr{C})$,
$\gamma(B)$ gives the number of cracks that initiated with
time-location instances $c=(t,x)\in B$. Note that
$\partial\Omega\subset \bar \Omega$, thus surface crack formation
-- as in the case of LCF -- can be modeled by measures $\gamma$
with support in $\mathscr{T}\times \partial\Omega\subseteq
\mathscr{C}$.

The occurrence of the first crack on $\bar \Omega$ will be
interpreted as a failure event. For the crack initiation history
$\gamma$, we define the failure time
$\tau:\mathscr{R}_c\to\mathscr{T}^\bullet=\mathscr{T}\cup\{\infty\}$
as
\begin{equation}
\tau(\gamma)=\inf\{ t>0: \gamma(\mathscr{C}_t)>0\},
\end{equation}
where $\mathscr{C}_t=\{(\tau,x)\in  \mathscr{C}:\tau\leq t\}$.
For $B\in\mathscr{B}(\mathscr{C})$ with bounded diameter, the
restriction of $\gamma\in \mathscr{R}_c$ to $B$ can be written as
(see \cite[Chapter 2]{Kallenberg})
\begin{equation}
\label{decompB}
\gamma\restriction_B=\sum_{j=1}^nb_j\delta_{c_j},~~c_j\in\mathscr{C}, c_i\not=c_j\mbox{ for }i\not=j , ~b_j\in\N.
\end{equation}
This decomposition is unique up to order. $\delta_c$ stands for
the Dirac measure in $c$. The Radon counting measure $\gamma$ is
called simple, if $\forall B\in\mathscr{B}(\mathscr{S})$ in
(\ref{decompB}) we have $b_j=1$ for all $j=1,\ldots,n$. The
simplicity of crack initiation histories $\gamma$ is a natural
condition, namely two cracks that originate at the same time at
the same place are considered as the same crack.

Next we have to take into account that crack initiation is a
random process. By $\mathscr{N}(\mathscr{R}_c)$ we denote the
standard sigma algebra on the space of Radon counting measures
generated by the mappings $\gamma\mapsto \int_{\mathscr{C}}f\,
d\gamma$ with $f\in C^0(\mathscr{C})$, the space of compactly
supported continuous functions on $\mathscr{C}$. The failure time
$\tau
:(\mathscr{R}_c,\mathscr{N}(\mathscr{R}_c))\to(\mathscr{T},\mathscr{B}(\mathscr{T}))$
is easily seen to be measurable. The following definition can be
found in \cite{Kallenberg}:

\begin{definition}
\label{pointprocess} \textbf{(Point Process)}\\
{\rm Let
$(\mathscr{X},\mathscr{A},P)$ be a probability space.
\begin{itemize}
\item[(i)] A point process on $\mathscr{C}$ is a measurable mapping $\gamma :(\mathscr{X}\mathscr{A},P)\to(\mathscr{R}_c,\mathscr{N}(\mathscr{R}_c))$.
\item[(ii)] The point process $\gamma$ is simple, if $\gamma(.,\omega)$ is simple  for $P$-almost all $\omega\in\mathscr{X}$.
\item[(iii)] A point process $\gamma$ is non-atomic, if $P(\gamma(\{c\})>0)=0$ $\forall c\in\mathscr{C}$.
\item[(iv)] A point process $\gamma$ has independent increments, if for $B_1,\ldots,B_n\in\mathscr{B}(\mathscr{C})$ mutually disjoint, the random variales $\gamma(B_1),\ldots,\gamma(B_n)$ are independent.
\end{itemize}
}
\end{definition}
Random crack initiation histories are naturally modeled as simple
point processes. The interpretation of the additional assumption
that $\gamma$ does not posses 'atoms' is that there is no location
$x\in\bar\Omega$ such that a crack will originate exactly in $x$
with a probability larger than zero.

\begin{definition}
\label{crackformationprocess} \textbf{(Crack Initiation Process)}
\begin{itemize}
\item[(i)] A crack initiation process $\gamma$ is a simple, non-atomic point process on $\mathscr{C}$.
\item[(ii)] The time to crack initiation $T:\mathscr{X}\to\mathscr{T}^\bullet$ associated with $\gamma$ is the random variable $T=\tau(\gamma)$.
\end{itemize}
\end{definition}

Whether the assumption of independent increments is realistic for
random crack initiation can be disputed. This approach however
should be valid if we are only interested in the component's
history until the formation of the first crack initiation. For a
model of interacting crack networks, see e.g. \cite{MVH}.

We now apply some standard results from the theory of point
processes to the present context:

\begin{proposition}
\label{Poisson_Point_Process} \textbf{(Crack Initiation Processes
and Poisson Point Processes)}
\begin{itemize}
\item[(i)] Any crack initiation process $\gamma$ on $\mathscr{C}$ with independent increments is a Poisson point process,
i.e.\ there exists a unique Radon measure $\rho\in\mathscr{R}$ such that
$$
P(\gamma(B)=n)= e^{-\rho(B)}\frac{\rho(B)^n}{n!} ~~\forall B\in\mathscr{B}(\mathscr{C})~\mbox{bounded}.
$$
$\rho$ is called the intensity measure of $\gamma$.

\item[(ii)] The distribution function $F_T$ of the time to crack initiation $T$ is given by $F_T(t)=1-e^{-H(t)}$ with
cumulative hazard function $H(t)=\rho(\mathscr{C}_t)$.

\item[(iii)] If $\rho(\mathscr{C})=\infty$, then $P(T=\infty)=0$ and $T$ can be modified to
$T:(\mathscr{X},\mathscr{A})\to (\mathscr{T},\mathscr{B}(\mathscr{T}))$.
\end{itemize}
\end{proposition}
\begin{proof}
Assertion (i) is proven in \cite[Corrollary 6.7]{Kallenberg} and
(ii) then follows from
$F_T(t)=1-P(T>t)=1-P(\gamma(\mathscr{C}_t)=0)=1-e^{-\rho(\mathscr{C}_t)}=1-e^{-H(t)}$.

Finally, consider (iii): If $\rho(\mathscr{C})=\infty$, we have by
lower continuity of radon measures that $H(t)\to\infty$ and thus
$S(t)=e^{-H(t)}=P(T>t)\rightarrow 0$ as $t\to\infty$. This implies
by lower continuity of $P$ that $T<\infty$ holds $P$ almost sure
and we can redefine $T=0$ on the null set $\{T=\infty\}$ without
changing the probability law of $T$.\qed
\end{proof}

The reliability of the component $\Omega$ at some warranty time
$t^*$ or after the passage of a service interval of duration $t^*$
depends on the loads that act on $\Omega$, the material and shape
$\Omega$ itself. As in many design applications the loads and the
material are given, the choice of the shape $\Omega$ is the
crucial design task. A natural question arising is: Is there a
shape with optimal reliability? The answer crucially depends on an
assignment of failure probabilities --  i.e.\ probability of crack
initiation on $\Omega$ until $t^*$ -- to the shape $\Omega$. The
following definition collects some basic requirements:

\begin{definition}
\label{reliability_optimization_problem} \textbf{(Crack Initiation Model)}\\
{\rm Let ${\cal O}$ be some collection of admissible domains
contained in $\Omega^{\rm ext}\subseteq \R^3$ and let
$f,g:\mathscr{C}^{\rm ext}=\mathscr{T}\times\Omega^{\rm ext}\to
\R^3$ be vector fields of some spaces  $\mathscr{V}_{\rm vol}$ and
$\mathscr{V}_{\rm sur}$, respectively. For $\Omega\in{\cal O}$,
$g\restriction_{\partial\Omega}$ is interpreted as the history --
or load collective -- of surface force densities on
$\partial\Omega$ and $f\restriction_\Omega$ as the histroy of
volume force densities on $\Omega$.

A crack initiation model is a mapping $\gamma$ from ${\cal
O}\times\mathscr{V}_{\rm vol}\times \mathscr{V}_{\rm sur} $ to the
space of all crack initiation processes on $\mathscr{T}\times
\Omega^{\rm ext}$ mapping $(\Omega,f,g)$ to $\gamma_{\Omega,f,g}$
such that
\begin{itemize}
\item[(i)]  $\gamma_{\Omega,f,g}(\mathscr{T}\times (\Omega^{\rm ext}\setminus\bar\Omega))=0$ $P$-almost surely;
\item[(ii)]   $\gamma_{\Omega,f,g}$ $P$-almost surely depends
only on $f\restriction_{\mathscr{T}\times\Omega}$ and $g\restriction_{\mathscr{T}\times \partial\Omega}$.
\end{itemize}
}
\end{definition}

Obviously, for any crack initiation model, we obtain an induced
mapping of $(\Omega,f,g)$ to  the associated crack initiation time
random variables $T_{\Omega,f,g}$ associated with
$\gamma_{\Omega,f,g}$. We can now define the optimal reliability
problem, given the fixed load histories $f$ and $g$ and a fixed
time $t^*\in\mathscr{T}$. For notational simplicity, we suppress
$f$ and $g$ dependency in the following.

\begin{definition}
\label{optimal_reliability_problem} \textbf{(Optimal Reliability Problem)}\\
{\rm Given $t^*\in\mathscr{T}$, $f\in\mathscr{V}_{\rm vol}$,
$g\in\mathscr{V}_{\rm sur}$ and a crack initiation model $\gamma$,
find $\Omega^*\in {\cal O}$ such that
$$
P(T_{\Omega^*}\leq t^*)\leq P(T_{\Omega}\leq t^*)~~\forall \Omega\in{\cal O}.
$$
}
\end{definition}

Now, we want to construct crack initiation models with independent
increments based on the PDE of linear isotropic elasticity. This
establishes the link between the optimal reliability problem of
Definition \ref{optimal_reliability_problem} and PDE constrained
shape optimization. For simplicity, we restrict ourselves to the
case that the load vector fields $f$ and $g$ are independent of
$t$ such that our model is based on one well defined load cycle,
see Figure \ref{Figure_cyclic_loading}. The time $t$ then counts
the number of such load cycles.


Let $\nu$ be the outward normal of the boundary $\partial\Omega$
and let $\partial\Omega=\partial\Omega_D\cup\partial\Omega_N$ be a
partition where $\partial\Omega_D$ is clamped and on
$\partial\Omega_N$ a force surface density
$g\restriction_{\partial\Omega_N}$ is imposed. Then, according to
\cite{Finite_Elements_Ern} the mixed problem of linear isotropic
elasticity is described by:
\begin{align}\label{Thermoelasticity.1.1}
\begin{split}
&\nabla\cdot\sigma(u)+f=0\hspace{32mm}\textrm{in }\Omega,\\
&\sigma(u)=\lambda(\nabla\cdot u)I+\mu(\nabla u+\nabla u^T)\quad\textrm{in }\Omega,\\
&u=0\hspace{51mm}\textrm{on }\partial\Omega_D,\\
&\sigma(u)\cdot \nu=g\hspace{41mm}\textrm{on }\partial\Omega_N.
\end{split}
\end{align}
Here, $\lambda>0$ and $\mu>0$ are the Lam\'e coefficients and
$u:\Omega\to\R^3$ is the displacement field on $\Omega$. $I$ is
the identity on $\R^3$. The linearized strain rate tensor
$\varepsilon(u):\Omega\rightarrow\mathbb{R}^{3\times 3}$ is
defined as $\varepsilon(u)=\frac{1}{2}(\nabla u+\nabla u^T)$.
Approximate numerical solutions can be computed by a finite
element approach, confer \cite{Hetnarski} and
\cite{Finite_Elements_Ern}.

\begin{definition}
\label{stress_models} \textbf{(Local Crack Initiation Model)}

\noindent {\rm Let ${\cal O}\times \mathscr{V}_{\rm
vol}\times\mathscr{V}_{\rm sur}$ be such that for all
$f\in\mathscr{V}_{\rm vol}$, $g\in\mathscr{V}_{\rm sur}$ and
$\Omega\in{\cal O}$ there exists a unique (weak) solution
$u(\Omega)$ to (\ref{Thermoelasticity.1.1}). Let furthermore
$\varrho_{\rm sur}:\mathscr{T}\times \R^d\to \bar\R_+$ and
$\varrho_{\rm vol}:\mathscr{T}\times \R^d\to \bar\R_+$ with
$d=3+\sum_{j=0}^r3^{j+1}=3+\frac{3}{2}(3^{r+1}-1)$ be measurable,
non negative functions with $r\in \N$ the order of the model.
Suppose that the $l$-th weak derivative $\nabla^lu$ of $u$ are
measurable functions for $l=0,\ldots,r$ and that the trace
$\nabla^lu\restriction_{\partial\Omega}$ is well defined in the
sense of measurable functions. Here, $(\nabla^l u)_{j_1\cdots
j_l}^j$ stands for $\frac{\partial^l u_j}{\partial x_{j_1}\cdots
\partial x_{j_l}}$. Then, we define:

\begin{itemize}
\item[(i)] A $r$-th order local crack initiation model with independent increments and linear elasticity state
equation is defined by this data by setting $\gamma_\Omega$ to be
the Poisson point process on $\mathscr{C}^{\rm ext}$ associated to
the intensity measure
\begin{eqnarray*}
\rho_\Omega(B)&=&\int_{B\cap (\mathscr{T}\times\Omega)}\varrho_{\rm vol}(t,x,u,\nabla u,\ldots,\nabla^ru)\, dtdx\\
&+&\int_{B\cap (\mathscr{T}\times\partial\Omega)}\varrho_{\rm sur}(t,x,u,\nabla u,\ldots,\nabla^ru)\, dtdA ~~\forall B\in\mathscr{B}(\mathscr{C}^{\rm ext}),
\end{eqnarray*}
provided that the resulting measures are Radon measures on
$\mathscr{C}^{\rm ext}$.
\item[(ii)] $\gamma$ is said to be strain driven, if $\varrho_{\rm sur}$ and $\varrho_{\rm vol}$ depend only on the elastic
 strain tensor field $\varepsilon(u)$. As the elastic stress tensor field $\sigma(u)$ can be obtained from $\varepsilon(u)$ and vice versa,
 strain and stress driven crack initiation models are synonymous.
\item[(iii)] If $\varrho_{\rm sur}=0$, then $\gamma$ is volume driven and if $\varrho_{\rm vol}=0$ it is surface driven.
\item[(iv)] We say that the $r$-th order crack initiation model has $s$-regular intensity functions, $s\in\mathscr{T}$,
if $\varrho_{\rm sur}$, $\varrho_{\rm vol}$ are in $C^0(\mathscr{T})\otimes C^s(\R^d)$.
\end{itemize}
}
\end{definition}

From the above definition it is now clear that the optimal
reliability problem  -- confer Definition
\ref{optimal_reliability_problem} -- in the case described $r$-th
order local crack initiation model is just a PDE constrained shape
optimization problem:

\begin{lemma} \textbf{(Optimal Reliability and PDE Constrained
Optimization)}\label{opt_reliability_is_shape_opt}\\
For a $r$-th order local crack initiation model with elasticity
state equation as defined in Definition \ref{stress_models}, the
optimal reliability problem given in Definition
\ref{optimal_reliability_problem} is equivalent to the shape
optimization problem given in Definition
\ref{definition_cost_functional_SO_problem} below with
\begin{align}\label{objectiveFunction}
\begin{split}
J(\Omega,y)=\int_{\Omega}{\cal F}_{\rm vol}(x,y,\nabla y,\ldots,\nabla^ry)\, dx\\
+\int_{\partial\Omega}{\cal F}_{\rm sur}(x,y,\nabla y,\ldots,\nabla^ry)\, dA
\end{split}
\end{align}
with ${\cal F}_{\rm vol}(\cdot)=\int_0^{t^*}\varrho_{\rm
vol}(t,\cdot)\, dt$ and ${\cal F}_{\rm
sur}(\cdot)=\int_0^{t^*}\varrho_{\rm sur}(t,\cdot)\, dt$ and
$y:\Omega^{\rm ext}\to\R^3$ sufficiently regular.

In particular, for the case of $s$-regular intensity functions, ${\cal F}_{\rm vol},{\cal F}_{\rm sur}\in C^s(\R^d)$.
\end{lemma}
\begin{proof}
By Proposition \ref{Poisson_Point_Process} (ii)  the optimal
reliability problem is equivalent to
$\rho_{\Omega^*}(\mathscr{C}_{t^*})\leq
\rho_{\Omega}(\mathscr{C}_{t^*})$ for all $\Omega\in{\cal O}$. Now
apply Definition \ref{stress_models} (i), Fubini's lemma for
positive functions and the properties of Bochner integrals. \qed
\end{proof}

\section{A Probabilistic Crack Initiation Model for LCF} \label{section_prob_lcf}
A somewhat general approach should contain at least one relevant
example. In the following we present probabilistic LCF crack
initiation as a surface and strain driven crack initiation model
with elasticity state equation. In \cite{SKRG1} this model has
been numerically implemented and applied to gas turbine design. An
extension to thermomechanical equations can be found in
\cite{SKRG2}. Experimental validation is presented in
\cite{SSBKRG}.

One of the paradoxes that comes with the use of the equations of
isotropic elasticity (\ref{Thermoelasticity.1.1}) in the context
of fatigue comes from the fact that any perfectly elastic
deformation of a material is completely reversible and therefore
does not lead to degradation. The time-honored elastic-plastic
stress conversion resolves this issue and is widely used in
engineering. We therefore choose it as the basis of our model.

Defining the von Mises stress as $\sigma_{v}=\sqrt{\frac{3}{2}{\rm
tr}(\sigma^{'2})}$, with $\sigma'=\sigma-\frac{1}{3}{\rm
tr}(\sigma)I$ the trace-free part of $\sigma$ capturing
non-hydrostatic stresses, only, we present the Ramberg-Osgood
equation which is used to locally derive strain levels from scalar
comparison stresses $\sigma_v$, confer \cite{Ramberg}. This
equation describes stress-strain curves of metals near their yield
points.

\textbf{Notation (Ramberg-Osgood Equation)}

\noindent Let $K>0$ denote the strain hardening coefficient and
$n'>0$ the strain hardening exponent. Then, the Ramberg-Osgood
relation between an elastic-plastic comparison strain
$\varepsilon^{\rm el-pl}\in \bar\R_+$ and an elastic-plastic
comparison stress $\sigma_v\in\bar \R_+$ is given by
\begin{align}\label{Thermoelasticity.8.1}
\varepsilon^{\rm el-pl}_v=\frac{\sigma_v^{{\rm el-pl}}}{E}+\left(\frac{\sigma^{{\rm el-pl}}_v}{K}\right)^{1/n'}
\end{align}
with Young's modulus $E=\frac{\mu(3\lambda+2\mu)}{\lambda+\mu}$.
The equation defines the comparison strain $\varepsilon_v^{\rm
el-pl}$. We also write $\varepsilon^{{\rm
el-pl}}_v=\textrm{RO}(\sigma^{{\rm el-pl}}_v)$.

The obvious problem is that the elastic-plastic comparison stress $\sigma^{{\rm el-pl}}_v$ needs to
be defined on the basis of the elastic von Mises stress $\sigma_v$
which can be approximated by a finite element calculation.
This is accomplished by the method of stress
shakedown by Neuber\footnote{The stress shakedown is based on an
energy-conservation ansatz, confer \cite{Hoffmann} for
details on Neuber shakedown in conjunction with equivalent
stresses. As an alternative to Neuber's rule we could have also
used Glinka's method, see e.g. \cite{Knop_Jones_Molent_Wang}.}, confer \cite{Neuber} and
\cite{Harders_Roesler}.

\textbf{Notation (Elastic-Plastic Stress Conversion and
Shakedown)}

\noindent Given $\sigma_v\in\bar\R_+$, the associated
elastic-plastic comparison stress $\sigma_v^{{\rm el-pl}}$ is
defined as the positive solution to the following equation
\begin{align}\label{Thermoelasticity.9.1}
\frac{(\sigma_v)^2}{E}=\sigma^{{\rm el-pl}}_v\,\varepsilon^{{\rm el-pl}}_v=\frac{(\sigma^{{\rm el-pl}}_v)^2}{E}+\sigma^{{\rm el-pl}}_v\left(\frac{\sigma^{{\rm el-pl}}_v}{K}\right)^{1/n'}.
\end{align}

From the the elastic von Mises comparison stress $\sigma_v$, we
can thus calculate the elastic-plastic von Mises stress
$\sigma^{{\rm el-pl}}_v$ by solving (\ref{Thermoelasticity.9.1})
and thus we are able to obtain $\varepsilon_v^{{\rm el-pl}}$ from
(\ref{Thermoelasticity.8.1}). We also write $\sigma_v^{{\rm
el-pl}}=\textrm{SD}(\sigma_v)$.

In structural analysis, fatigue describes the damage or failure of
material under cyclic loading, confer \cite{Harders_Roesler} and
\cite{Vormwald}. Compared to the static case material is damaged
by much lower load amplitudes of cyclic loading. Figure
\ref{Figure_cyclic_loading} shows a triangle shaped uniaxial
load-time-curve as an example, where
$\sigma_a=[(\sigma_{\max}-\sigma_{\min})/2]_v$ is the elastic von
Mises comparison stress amplitude. For more details on surface
driven LCF failure mechanism with respect to polycrystalline
metal, we refer to
\cite{Harders_Roesler,Vormwald,Fedelich,Sornette}.  In fatigue the
number of cycles until failure is determined and if the tests are
strain controlled so-called $E-N$ diagrams can be created, see the
test points in Figure \ref{Figure_woehler_diagram_strain}.
\begin{figure}
\begin{minipage}[hbt]{8cm}
    \centering
\scalebox{0.33}{\input{Doktorarbeit_cyclic_loading_2.pstex_t}}
\caption{Triangle shaped load-time-curve.}
    \label{Figure_cyclic_loading}
\end{minipage}
\hfill
\begin{minipage}[hbt]{8cm}
    \centering
\scalebox{0.28}{\input{Doktorarbeit_EN_diagram_2.pstex_t}}
\caption{EN-diagram of a standardized specimen.}
    \label{Figure_woehler_diagram_strain}
\end{minipage}
\end{figure}

Figure \ref{Figure_woehler_diagram_strain} also shows the
relationship of strain amplitude $\varepsilon_a^{\rm el-pl}={\rm
RO}\circ{\rm SD}(\sigma_a)$ and the life time $N_i$ to crack
initiation measured in cycles. The Coffin-Manson-Basquin (CMB) or
W\"ohler equation\footnote{ A discussion of the physical origin of
this equation can be found in \cite{Sornette}.} connects the
number of cycles to crack initiation $N_i$ with the plastic strain
amplitude $\varepsilon_a^{\rm el-pl}$.

\textbf{Notation (Coffin-Manson-Basquin Equation)}

\noindent The CMB-equation connects the (deterministic) time to
crack initiation $N_i$ with the elastic-plastic strain amplitude
$\varepsilon_a^{\rm el-pl}$ via
\begin{equation}\label{materials_1.1}
\varepsilon_a^{{\rm el-pl}}={\rm CMB}(N_i)=\frac{\sigma_f'}{E}(2N_i)^b+\varepsilon_f'(2N_i)^c.
\end{equation}
Here, $\sigma_f'$ and $\epsilon_f'$ are positive and $b$ and $c$
negative material parameters.

In the following, we will assume for simplicity that the lower
edge of the load cycle is stress-free, corresponding to $f_{\rm
min}=0$ and $g_{\rm min}=0$ in (\ref{Thermoelasticity.1.1}) and
thus $\sigma_a=\sigma_v/2$, where we set $\sigma=\sigma_{\rm
max}$, $f=f_{\rm max}$ and $g=g_{\rm max}$.

In deterministic design, the lifetime of a component under cyclic
loading corresponds to the loading conditions at the part's
surface position of highest stress. Safety factors are
additionally imposed to account for the stochastic nature of LCF
and size effects\footnote{Note that different geometries of test
specimens lead to different W\"ohler /CMB curves, confer
\cite{Vormwald}.}. We refer to this method as the safe-life
approach in fatigue design which is widely used in engineering,
confer \cite{Harders_Roesler,Vormwald,SSBKRG,Schott}. The
following lemma collects the mathematical properties of this
procedure that will be of importance:

\begin{lemma}
\label{PropertiesNi} $\varphi={\rm CMB}^{-1}\circ{\rm RO}\circ
{\rm SD}:\R_+\to \R_+$, that maps elastic von Mises comparison
stress to a predicted life time, satisfies:
\begin{itemize}
\item[(i)] $\varphi$ is bijective and strictly monotonically decreasing;
\item[(ii)] $\lim_{\sigma_v\to0}\varphi(\sigma_v)=\infty$;
\item[(iii)] $\varphi$ is in $C^\infty(\R_+)$.
\end{itemize}
\end{lemma}
\begin{proof}
It is elementary to check that ${\rm RO}:\R_+\to\R_+$ is
bijective, an element of $C^\infty(\R^+)$ and strictly monotonic
increasing. As inverse of strictly monotonically increasing
(decreasing) $C^\infty(\R_+)$-function, the same applies to ${\rm
SD}$ (${\rm CMB}^{-1}$, which is strictly monotonically
decreasing, however). The three assertions of the lemma now follow
immediately. \qed
\end{proof}

\textbf{Notation (Deterministic LCF-Life at a Surface Point)}

\noindent Let $\mathbb{R}^{3\times 3}$ be the space of real
$3\times3$ matrices. Define $N_{\rm det}:\mathbb{R}^{3\times
3}\to\R_+\cup\{\infty\}$ via
\begin{equation}
\label{eqaNdet}
N_{\rm det}(M)=\varphi(\left[\lambda \, {\rm tr}(M)I+\mu(M+M^T)\right]_v),
\end{equation}
with $\varphi$ as in Lemma \ref{PropertiesNi} extended by
$\varphi(0)=\infty$. Given a solution $u\in [H^1(\Omega)]^3$ to
(\ref{Thermoelasticity.1.1}) such that the trace $\nabla
u\restriction_{\partial\Omega}$ can be reasonably defined and can
be represented as a continuous function, $N_{\rm det}(\nabla
u(x))$ is the predicted deterministic LCF-life at the point
$x\in\partial\Omega$.

Although somewhat out of the main line of this article, we state
the following definition that represents (up to safety factors)
the traditional engineering lifing approach:

\begin{definition}\label{definition_deterministic_LCF} \textbf{(Deterministic Optimal LCF-Lifing Problem)}\\
\emph{Let the setting of the previous Definition of LCF-life at a surface point be given.
\begin{itemize}
\item[(i)] The deterministic life time to crack initiation
for LCF is defined as
$$T_{{\rm det},\Omega}=\inf\{N_{\textrm{det}}(\nabla u(x))|x\in\partial\Omega\}.$$
\item[(ii)] Given a set of admissible shapes ${\cal O}$ and loads
$f\in\mathscr{V}_{\rm vol}$, $g\in \mathscr{V}_{\rm sur}$ such that the above setting
holds for all $\Omega\in{\cal O}$, we define the deterministic optimal LCF lifing problem as follows
$$
\mbox{Find } \Omega^*\in {\cal O} \mbox{ such that }T_{{\rm det},\Omega^*}\geq T_{{\rm det},\Omega}~\forall\Omega\in{\cal O}.
$$
\end{itemize}}
\end{definition}

It is a usual approach in reliability statistics
\cite{Escobar_Meeker}, to choose the deterministic life prediction
as a scale variable of a failure time distribution. Moreover,
Weibull distributions are widely used in technical reliability
analysis. Recall that the Weibull distribution with scale
parameter $\eta$ and shape parameter $m$ is defined by the
cumulative distribution function
$F(t)=1-e^{-\left(\frac{t}{\eta}\right)^m}$ for $t>0$ and zero
otherwise.

\begin{definition}\label{definition_local_probabilistic_LCF} \textbf{(Local Weibull Model for LCF)}\\
\rm For $m\geq 1$, the strain and surface driven crack initiation
model (confer Definition \ref{stress_models}) with independent
increments of $1$-st order that is defined by
\begin{align}\label{materials_3.1}
\begin{split}
\varrho_{\rm vol}=0,~~\varrho_{\rm sur}(t,\nabla u)=\frac{m}{N_{\textrm{det}}(\nabla u)}\left(\frac{t}{N_{\textrm{det}}(\nabla u)}\right)^{m-1},
\end{split}
\end{align}
is called the local (probabilistic) Weibull model for LCF. The
associated optimal reliability problem as given in Definition
\ref{optimal_reliability_problem} and Lemma
\ref{opt_reliability_is_shape_opt} is called the optimal
reliability problem for LCF. Here, we employ the convention
$\frac{1}{\infty}=0$.
\end{definition}

The model could be defined for both, $\mathscr{T}=\N_0$ and
$\mathscr{T}=\bar\R_+$, but the second option is used more often.
In this case, ${\cal F}_{\rm sur}(\nabla u)=\left(\frac{t}{N_{\rm
det}(\nabla u)}\right)^m$.

\begin{proposition}\label{properties_Weibull_model} \textbf{(Properties of the local Weibull Model for LCF)}

\noindent Let the conditions of Definition \ref{stress_models} be
fulfilled such that $\left\|\frac{1}{N_{\rm det}(\nabla
u)}\right\|_{L^m(\partial\Omega)}<\infty$ for some $m>1$, and for
all $f\in\mathscr{V}_{\rm vol}$, $g\in\mathscr{V}_{\rm sur}$ and
$\Omega\in{\cal O}$. Then,
\begin{itemize}
\item[(i)] The local probabilistic Weibull model for LCF actually induces a 1st order local crack initiation
model, i.e. the associated measures $\rho_\Omega$ are Radon
measures.
\item[(ii)] The intensities of this model are $0$-regular, i.e. are continuous functions of $\nabla u$.
\item[(iii)] The crack initiation time $T_\Omega$ is Weibull distributed with shape parameter $m$
and scale parameter  $\eta=\left\|\frac{1}{N_{\rm det}(\nabla u)}\right\|_{L^m(\partial\Omega)}^{-1}$.
\end{itemize}
\end{proposition}
\begin{proof}
(i) To see the finiteness of the intensity measure on sets of
finite diameter, note that for any such set $B\subseteq
\mathscr{B}(\mathscr{C}^{\rm ext})$ we find $t\in\mathscr{T}$ such
that $B\subseteq \mathscr{C}^{\rm ext}_t$ and thus, with ${\cal
F}_{\rm sur}(\nabla u)=\left(\frac{t}{N_{\rm det}(\nabla
u)}\right)^m$, confer Definition
\ref{definition_local_probabilistic_LCF},
\begin{equation}
\label{WeibullProbLCF}
\rho_\Omega(B)\leq\rho_\Omega(\mathscr{C}^{\rm ext}_t)=\int_{\mathscr{C}^{\rm ext}_t\cap(\mathscr{T}\times \partial\Omega)}\varrho_{\rm sur}(t,\nabla u)\,dAdt=t^m\left\|\frac{1}{N_{\rm det }(\nabla u)}\right\|_{L^m(\partial\Omega)}^m <\infty.
\end{equation}

(ii) Obviously, the elastic von Mises stress depends continuously
on the components of $\nabla u$. Taking into account the
convention $\frac{1}{\infty}=0$ and Lemma \ref{PropertiesNi}, we
see that $\frac{1}{N_i(\nabla u)}$ depends continuously on the
components of $\nabla u$. This implies $0$-regularity of the
intensity function (\ref{materials_3.1}) in the sense of
Definition \ref{stress_models} (iv).

(iii) The right hand side of (\ref{WeibullProbLCF}) equals the
cumulative hazard function $H(t)$ of $T_\Omega$. Comparing this to
the Weibull cumulative hazard function $H_{\rm
Wei}(t)=\left(\frac{t}{\eta}\right)^m$, the assertion follows.
\qed
\end{proof}

Higher order regularity of the intensity functions crucially
depends on the material parameters. We postpone a detailed account
to \cite{Gottschalk_Schmitz}.

\section{Shape Optimization and $C^k$-Admissible Domains}\label{section_SO_admissible_domains}

In this section, we introduce into an abstract setting of shape
optimization and discuss so-called $C^k$-admissible domains. 
This leads to a theoretical frame for an existence proof of
optimal designs in Section \ref{section_optimal_reliability} with
respect to a general class of cost functionals which include some
functionals given in Lemma \ref{opt_reliability_is_shape_opt},
namely those coming from $3$-rd order local crack initiation model
with $0$-regular intensity functions, see  Definition
\ref{stress_models} (i--iv). In particular this includes the
local, probabilistic Weibull model for LCF, confer Definition
\ref{definition_local_probabilistic_LCF} and Proposition
\ref{properties_Weibull_model}, but also the deterministic optimal
LCF lifing problem, confer Definition
\ref{definition_deterministic_LCF}.

First, we introduce basic notations and closely follow Section 2.4 in
\cite{Shape_Optimization_Haslinger}. These notations have to be precisely concretized in each class
of shape optimization problems, as we will do so at the end of this section and in
Section \ref{section_optimal_reliability}, after we have established the $C^k$-admissible domains and
the state problem, respectively. For further introductions
into shape optimization confer \cite{Sokolowski_Zolesio,Bucur_Buttazzo,DZ}, for example.

\textbf{Notation (Family of Admissible Domains, State Space)}\\
Let $\tilde{\mathcal{O}}$ denote a family of admissible domains
and let $V(\Omega)$ for every $\Omega\in\tilde{\mathcal{O}}$
denote a state space of real functions defined in $\Omega$.

\textbf{Notation (Convergence of Sets and of Functions
with Variable Domains)}\\
Let $(\Omega_n)_{n\in\mathbb{N}}$ be a sequence in
$\tilde{\mathcal{O}}$ and let $\Omega\in\tilde{\mathcal{O}}$. Then
$\Omega_n\stackrel{\tilde{\mathcal{O}}}{\longrightarrow}\Omega$ as
$n\rightarrow\infty$ denotes the convergence of
$(\Omega_n)_{n\in\mathbb{N}}$ against $\Omega$. If
$(y_n)_{n\in\mathbb{N}}$ is a sequence of functions with $y_n\in
V(\Omega_n)$ for every $n\in\mathbb{N}$ and if $y\in V(\Omega)$
then $y_n\rightsquigarrow y$ as $n\rightarrow\infty$ denotes the
convergence of $(y_n)_{n\in\mathbb{N}}$ against $y$. Moreover, it
is assumed that any subsequence of a convergent sequence converges
against the limit of the original one.

In every $\Omega\in\tilde{\mathcal{O}}$ one solves a state problem
which can be a PDE or a variational inequality, for example.
Assuming that every state problem has a unique solution and
associating with any $\Omega\in\tilde{\mathcal{O}}$ the
corresponding unique solution $u(\Omega)\in V(\Omega)$ one obtains
the map $u:\Omega\longmapsto u(\Omega)\in V(\Omega)$. Let
$\mathcal{O}$ be a subfamily of $\tilde{\mathcal{O}}$, then
$\mathcal{G}=\{(\Omega,u(\Omega))\,|\,\Omega\in\mathcal{O}\}$ is
called the graph of the mapping $(u(\cdot))$ restricted to
$\mathcal{O}$.

\begin{definition}\label{definition_cost_functional_SO_problem}\textbf{(Cost Functional, Optimal Shape Design Problem)}\rm\\
A cost functional $J$ on $\tilde{\mathcal{O}}$ is given by a map
$J:(\Omega,y)\longmapsto J(\Omega,y)\in \mathbb{R}$, where
$\Omega\in\tilde{\mathcal{O}}$ and $y\in V(\Omega)$. Let
$\mathcal{O}$ be a subfamily of $\tilde{\mathcal{O}}$ and for
every $\Omega\in\mathcal{O}$ let $u(\Omega)$ be the unique
solution of a state problem given in $\Omega$. An optimal shape
design problem can then be defined by
\begin{equation}\label{SO.1.1}
 \left\{%
 \begin{array}{l}
   \textrm{Find $\Omega^*\in\mathcal{O}$ such that} \\
   J(\Omega^*,u(\Omega^*))\leq J(\Omega,u(\Omega))\quad\forall\Omega\in\mathcal{O}. \\
 \end{array}%
 \right.
\end{equation}
\end{definition}

Now, we present a statement regarding the existence of optimal
shapes. Note that this theorem will structure the following
section where we prove of our main existence results.

\begin{theorem}\label{theorem_optimum_shape_design} \textbf{(Existence of An Optimum in Shape Design Problems)}\\
Let $\tilde{\mathcal{O}}$ be a family of admissible domains and
$\mathcal{O}$ a subfamily. Moreover, let $J$ be a cost functional
on $\tilde{\mathcal{O}}$ and assume that every
$\Omega\in\tilde{\mathcal{O}}$ has a state problem with state
space $V(\Omega)$ where each such state problem has a unique
solution $u(\Omega)\in V(\Omega)$. Finally, conjecture
\begin{itemize}
\item[(i)] Compactness of $\mathcal{G}=\{(\Omega,u(\Omega))\,|\,\Omega\in\mathcal{O}\}$:

Every sequence
$(\Omega_n,u(\Omega_n))_{n\in\mathbb{N}}\subset\mathcal{G}$ has a
subsequence $(\Omega_{n_k},u(\Omega_{n_k}))_{n_k\in\mathbb{N}}$
which satisfies
\begin{align*}
  &\Omega_{n_k}\stackrel{\tilde{\mathcal{O}}}{\longrightarrow}\Omega,\quad k\rightarrow\infty,\\
&u(\Omega_{n_k})\rightsquigarrow u(\Omega),\quad
k\rightarrow\infty
\end{align*}
for some $(\Omega,u(\Omega))\in\mathcal{G}$.
\item[(ii)] Lower semi-continuity of $J$:

Let $(\Omega_n)_{n\in\mathbb{N}}$ with
$\Omega_n\in\tilde{\mathcal{O}}$, $n\in\mathbb{N}$, and
$(y_n)_{n\in\mathbb{N}}$ with $y_n\in V(\Omega_n)$,
$n\in\mathbb{N}$, be sequences and let $\Omega$ and $y$ be some
elements in $\tilde{{\mathcal{O}}}$ and in $V(\Omega)$,
respectively. Then
\begin{align*}
 \left.%
 \begin{array}{l}
  \Omega_{n}\stackrel{\tilde{\mathcal{O}}}{\longrightarrow}\Omega,\quad n\rightarrow\infty, \\
  y_n\rightsquigarrow y,\quad n\rightarrow\infty  \\
 \end{array}%
 \right\}
 \left.%
 \begin{array}{l}
  \\
\quad  \Longrightarrow\quad\liminf_{n\rightarrow\infty} J(\Omega_{n},y_{n})\geq J(\Omega,y). \\
  \\
 \end{array}%
 \right.
\end{align*}
\end{itemize}
Under these assumptions, the optimal shape design problem
(\ref{SO.1.1}) possesses at least one solution.
\end{theorem}

According to the previous theorem and abstract setting we have to
define the family of admissible domains $\tilde{\mathcal{O}}$. In
this work we consider so-called $C^k$-admissible domains which
have smooth boundaries. For the sake of simplicity, we only
optimize a part of the boundary of these shapes. This method is
described in Section 2.8 of \cite{Sokolowski_Zolesio} and in
\cite{Shape_Optimization_Haslinger}, too. Importantly,
$C^k$-admissible domains satisfy compactness properties as
required in Theorem \ref{theorem_optimum_shape_design} according
to the Arzela-Ascoli theorem.

At first, admissible domains are defined which are determined by
uniformly bounded functions. Later on, these functions are assumed
to be sufficiently smooth to obtain $C^k$-admissible domains. On
these domains we will later impose boundary value problems (BVPs)
of linear elasticity which are characterized by disjoint Dirichlet
and Neumann boundaries. In the following we employ so-called
uniform cone properties which are defined in Appendix A of
\cite{Shape_Optimization_Haslinger}.

\begin{figure}[htbp]
\centering
\scalebox{0.37}{\input{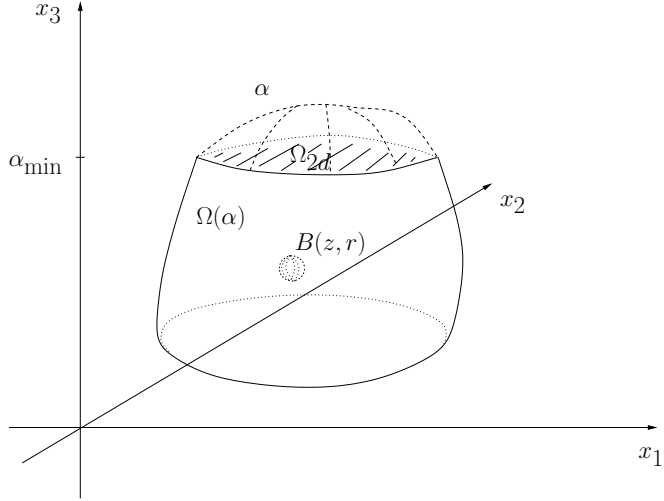}}
\caption{Admissible shape.}
    \label{Figure_admissible_shape}
\end{figure}

\begin{definition}\label{definition_design_variables_admissible_domains}\textbf{(Basic Design, Design Variables, Admissible Domains)}\rm\\
Let $\hat{\Omega}\subset\mathbb{R}^3$ be a simply connected and
bounded domain which is called basic design under the following
assumptions:

\begin{itemize}
\item[(i)] $\hat{\Omega}$ has a uniform cone property and a $C^k$-boundary\footnote{This leads to a
local description of $\partial\hat{\Omega}$ by a finite number of
hemisphere transformations of class $C^k$, confer Definitions
\ref{definition_hemisphere_property} and
\ref{definition_Lipschitz_Cka_boundary_Gilbarg} and Remark
\ref{remark_Cka_boundary}.} for some $k\in\mathbb{N},k\geq1$.
\item[(ii)] There is some $\alpha_{\min}\in\mathbb{R}$ so that the cross section
$\Omega_{2d}=\{(x_1,x_2)\in\mathbb{R}^2\,|\,(x_1,x_2,\alpha_{\min})\in\hat{\Omega}\}$
is a nonempty domain in $\mathbb{R}^2$.
\item[(iii)] There are some $z\in\hat{\Omega}$ and $r>0$ such
that $B(z,r)\subset\hat{\Omega}$ and $z_3+r<\alpha_{\min}$.
\end{itemize}
For $\alpha_{\max}>\alpha_{\min}$ and positive constants
$L_1,L_2,L_3$ the elements of
\begin{align*}
\begin{split}
&\tilde{U}^{\textrm{ad}}=\left\{\alpha\in
  C^{k}(\overline{\Omega}_{2d})\,\left|\,\alpha_{\min}\leq\alpha\leq\alpha_{\textrm{max}}\,\textrm{in
}\overline{\Omega}_{2d},\,\alpha|_{\partial\Omega_{2d}}=\alpha_{\min},\,\int_{\Omega_{2d}}\right.\alpha(x)dx=L_1,\,\|\alpha\|_{C^k}\leq
L_2,\right.\\
&\hspace{37mm}\left.\left|\alpha^{(k)}(x)-\alpha^{(k)}(y)\right|
\leq L_3\|x-y\|_2\,\,\forall x,y\in\overline{\Omega}_{2d}\right\}
\end{split}
\end{align*}
are called design variables. Let $\alpha\in \tilde{U}^{ad}$ define
the set
\begin{align*}
\begin{split}
\Omega(\alpha)=\{x\in\hat{\Omega}\,|\,x_3\leq\alpha_{\min}\}\setminus\overline{B(z,r)}\,\,\cup\,\,
\{x\in\mathbb{R}^3\,|\,(x_1,x_2)\in\Omega_{2d},\,\alpha_{\min}<
x_3<\alpha(x_1,x_2)\},
\end{split}
\end{align*}
see Figure \ref{Figure_admissible_shape}. Varying only functions
$\alpha\in \tilde{U}^{\textrm{ad}}$ the domains
$\tilde{\mathcal{O}}=\{\Omega(\alpha)\,|\,\alpha\in
\tilde{U}^{\textrm{ad}}\}$ with Lipschitz continuous boundaries
are called admissible domains. Finally, choose an open superset
for all admissible domains such as
$\Omega^{\textrm{ext}}=B(z,r^{\textrm{ext}})$ with
$r^{\textrm{ext}}>0$ sufficiently large.
\end{definition}

\begin{lemma}\label{lemma_compact_admissible_shapes}
$\tilde{U}^{\textrm{ad}}$ is compact in
$\left(C^k(\overline{\Omega}_{2d}),\|\cdot\|_{C^k}\right)$.
\end{lemma}
\begin{proof}
Applying the Arzela-Ascoli theorem several times to
\begin{align}\label{SO.1.2}
\left\{\alpha\in
  C^{k}(\overline{\Omega}_{2d})\,\left|\,
  \|\alpha\|_{C^k}\leq
L_2,\left|\alpha^{(k)}(x)-\alpha^{(k)}(y)\right|\leq
L_3\|x-y\|_2\,\,\forall x,y\in\overline{\Omega}_{2d}\right.\right\} 
\end{align}
shows the compactness in
$\left(C^k(\overline{\Omega}_{2d}),\|\cdot\|_{C^k}\right)$.
$\tilde{U}^{\textrm{ad}}$ being a closed subset of (\ref{SO.1.2})
proves the statement of this lemma. \qed
\end{proof}

\begin{definition}\label{definition_Ck_set_convergence}\textbf{($C^k$-Convergence of Sets)}\rm\\
$\Omega(\alpha_n)\stackrel{\tilde{\mathcal{O}}}{\longrightarrow}\Omega(\alpha)$
as $n\rightarrow\infty$
is defined by $\alpha_n\rightarrow\alpha$ in
$C^k(\overline{\Omega}_{2d})$ as $n\rightarrow\infty$, where
$\alpha,\alpha_n$ and $\Omega(\alpha),\Omega(\alpha_n)$ for
$n\in\mathbb{N}$ are defined
 as in Definition \ref{definition_design_variables_admissible_domains}.
\end{definition}

In order to apply regularity results of linear elasticity we have
to require a sufficiently smooth boundary of the admissible
domains. Therefore, additional boundary conditions on the design
variables $\alpha$ are introduced which enable the construction of
such domains.

\begin{definition}\label{definition_Ck_admissible_domains}\textbf{($C^k$-Admissible Domains)}\rm\\
Let $\tilde{U}^{\textrm{ad}}$ be the set of design variables of
Definition \ref{definition_design_variables_admissible_domains}
and let $S_\beta:\partial\Omega_{2d}\rightarrow\mathbb{R}$ be
functions for multi-indices $\beta$ with $1\leq |\beta|\leq k$.
Define
$U^{\textrm{ad}}=\left\{\alpha\in\tilde{U}^{\textrm{ad}}\,\left|\,\nabla^\beta\alpha|_{\partial\Omega_{2d}}=S_\beta\,\,\forall
|\beta|\,\in\{1,...,k\}\right.\right\}$. Choosing $S_\beta$ so
that $\Omega(\alpha)$ has a $C^k$-boundary for every $\alpha\in
U^{ad}$ the set $\mathcal{O}=\{\Omega(\alpha)\,|\,\alpha\in
U^{\textrm{ad}}\}$ denotes the family of so-called
$C^k$-admissible domains.
\end{definition}

\begin{lemma}\label{lemma_compact_Ck_admissible_shapes}
$U^{\textrm{ad}}$ is compact in
$\left(C^k(\overline{\Omega}_{2d}),\|\cdot\|_{C^k}\right)$.
\end{lemma}
\begin{proof}
Note that $U^{\textrm{ad}}$ is a closed subset of
$\tilde{U}^{\textrm{ad}}$, where $\tilde{U}^{\textrm{ad}}$ is
already compact according to Lemma
\ref{lemma_compact_admissible_shapes}. \qed
\end{proof}

\section{Uniform Schauder Estimates}\label{section_schauder_estimates}

In this section, we review regularity results in linear elasticity
and show that certain estimates that lead to the mentioned
regularity are uniform in ${\cal O}$. From the shape optimization
prospective, these results will influence our choice of state
space and of the definition of convergence of functions with
variable domains. But they also provide the necessary input in
order to define what is meant by sufficiently regularity in
Definition \ref{stress_models} and provide candidates for the
vector spaces $\mathscr{V}_{\rm vol/surf}$ in that definition for local
crack initiation models of order up to three.

Recall the mixed problem (\ref{Thermoelasticity.1.1}). Regularity
results for the mixed problem depend crucially on the properties
of the domain's boundary in which the elasticity equations are
posed. Theorem 6.3-5 of \cite{Elasticity_Ciarlet_1} ensures the
existence of a weak solution of the mixed problem:

\begin{theorem}\label{theorem_existence_weak_solution} \textbf{(Existence of a Weak Solution)}\\
Let $\Omega\subset\mathbb{R}^n$ be a domain with a Lipschitz
continuous boundary and let
$\partial\Omega_D\subset\partial\Omega$ be measurable where
$\partial\Omega_D$ has a positive area. Let the Lam\'e
coefficients $\lambda,\,\mu$ be positive constants and let $f\in
[L^{6/5}(\Omega)]^3$, $g\in[L^{4/3}(\partial\Omega_N)]^3$ where
$\partial\Omega_N=\partial\Omega\setminus\partial\Omega_D$.
Moreover, define on $V_{\textrm{\emph{DN}}}=\left\{v\in
[H^1(\Omega)]^3\,|\,v=0\textrm{ a.e. on }\partial\Omega_D\right\}$
\begin{align*}
&B(u,v)=\int_\Omega\lambda(\nabla\cdot u) (\nabla\cdot v)\,\,dx+\int_\Omega 2\mu\,{\rm tr}(\varepsilon(u)\varepsilon(v))dx,\\
&L(v)=\int_\Omega f\cdot v\,\,dx+\int_{\partial\Omega_N}g\cdot
v\,\,dA.
\end{align*}
Then, there exists a unique $u\in V_{\textrm{\emph{DN}}}$ that
satisfies
\begin{align}\label{Elasticity.1.2}
B(u,v)=L(v)\quad\forall v\in V_{\textrm{\emph{DN}}}
\end{align}
and additionally $J(u)=\inf\{J(v)\,|\,v\in
V_{\textrm{\emph{DN}}}\}$, where $J(v)=\frac{1}{2}B(v,v)-L(v)$.
\end{theorem}

The unique solution $u\in V_{\textrm{DN}}$ of
(\ref{Elasticity.1.2}) is called the weak solution of
(\ref{Thermoelasticity.1.1}). The following inequality can be
found in \cite{Braess}: If $\Omega\subset\mathbb{R}^3$ is a domain
the so-called Korn's second inequality
\begin{align}\label{theorem_Korns_first_inequality}
c\|v\|_{[H^1_0(\Omega)]^3}\leq\left(\int_\Omega{\rm tr}(\varepsilon(v)^2)dx\right)^{1/2}
=\|\varepsilon(v)\|_{H^0(\Omega)}
\end{align}
holds for all $v\in V_{\textrm{DN}}$. We now consider the
co-called disjoint displacement-traction problem of linear
elasticity (\ref{Thermoelasticity.1.1}), where
$\overline{\partial\Omega_D}\cap\overline{\partial\Omega_D}=\emptyset$.
This BVP will have additional regularity properties if
$\partial\Omega$ and the forces $f$ and $g$ are sufficiently
regular, confer Theorem 6.3-6 of \cite{Elasticity_Ciarlet_1} and
remarks thereafter:

\begin{theorem}\label{corollary_regularity_weak_solution}
\textbf{(Regularity for the Disjoint Displacement-Traction Problem)}\\
Let $\Omega\subset\mathbb{R}^3$ be a domain with a $C^4$-boundary,
let $f\in [W^{2,p}(\Omega)]^3$ and let $g\in
[W^{1-1/p,p}(\partial\Omega)]^3$ for some $p\geq6/5$. Consider on
$\Omega$ a disjoint displacement-traction problem. Then, there
exists a unique solution $u\in V_{\textrm{\emph{DN}}}$ of
$B(u,v)=L(v)$ for all $v\in V_{\textrm{\emph{DN}}}$, where
$V_{\textrm{\emph{DN}}},B$ and $L$ are defined as in Theorem
\ref{theorem_existence_weak_solution}. Moreover, $u$ is an element
of $[W^{4,p}(\Omega)]^3$.
\end{theorem}

The key to the proof is to employ the fact that the previous
problems (pure displacement, pure traction, disjoint
displacement-traction problem) are uniformly elliptic and satisfy
the so-called supplementary and complementing conditions. These
conditions are introduced in \cite{Agmon_Douglis_Nirenberg_2},
where Schauder estimates are also described in detail which can be
applied to solutions of mixed problems. The Schauder estimates are
an important ingredient in our proof of existence for optimal
shapes in this section.

From now on, we consider on $C^4$-admissible shapes\footnote{$C^4$
is needed for Theorem \ref{corollary_regularity_weak_solution}.
Therefore, set $k=4$ in the definition of $\mathcal{O}$.}
$\Omega(\alpha)\in\mathcal{O}$ the disjoint displacement-traction
problem of linear elasticity as follows:

\begin{definition}\label{state_problem}\textbf{(State Problem ${\cal P}(\alpha)$)}

 {\rm
\noindent The state problem ${\cal P}(\alpha)$ for a
$C^k$-admissible shape $\Omega(\alpha)\in{\cal O}$ is defined to
be given by the elasticity equation (\ref{Thermoelasticity.1.1})
with $\Omega=\Omega(\alpha)$ where
 $\partial\Omega(\alpha)_D=\partial B(z,r)$ is the complete
interior boundary,
$\partial\Omega(\alpha)_N=\partial\Omega(\alpha)\setminus\partial\Omega(\alpha)_D$
the exterior boundary\footnote{Note that this decomposition of the
boundary depends continuously on $\alpha\in U^{ad}$ and the
two-dimensional Lebesgue measure of $\partial\Omega(\alpha)_D$ is
greater than a positive constant for all $\alpha\in U^{ad}$.} and
where $\nu$ is the normal of $\partial\Omega_N(\alpha)$. See
Figure \ref{Figure_admissible_shape}, too. We choose
$V(\Omega(\alpha))=[C^{3}(\overline{\Omega(\alpha)})]^3$ as state
space\footnote{This allows to analyze third order local crack
initiation models.} for $\mathcal{P}(\alpha)$. }
\end{definition}

As the crucial step in this section, we use the Schauder estimates
of Theorem 9.3 in \cite{Agmon_Douglis_Nirenberg_2} and validate if
the corresponding assumptions are satisfied. At first, we present
two lemmas which show the existence of sufficiently regular
hemisphere transformations -- confer Definition
\ref{definition_hemisphere_property}  --  and the validity of a
certain inequality, respectively. These two lemmas will be used in
the proof for the next theorem. In the following statements Banach
spaces $C^{q,\phi}$ of H\"older continuous $C^q$-functions for
$\phi\in(0,1)$ occur, whose definition can e.g.\ be found in
Section 1 of \cite{Alt} and in Section 7 of
\cite{Agmon_Douglis_Nirenberg_2}.

\begin{lemma}\label{lemma_hemisphere_trafo}
Each $\Omega\in\mathcal{O}$ satisfies a hemisphere property where
the corresponding hemisphere transformations are of class
$C^{3,\phi}$ for $\phi\in(0,1)$ and have a uniform bound $\kappa$
with respect to $\mathcal{O}$.
\end{lemma}
\begin{proof}
Regarding Definition \ref{definition_hemisphere_property} we have
to show that every $x\in\Omega$ within a certain distance $d>0$ of
$\partial\Omega$ has a neighborhood $U_x$ with $B(x,d/2)\subset
U_x$ and
\begin{equation*}
\overline{U}_x\cap\overline{\Omega}=T_x(\Sigma_{R(x)}),\quad
    0<R(x)\leq1,\quad\overline{U}_x\cap\partial\Omega=T_x\left(F_{R(x)}\right),
\end{equation*}
for some hemisphere\footnote{$F_{R(x)}$ denotes the flat boundary
of the hemisphere $\Sigma_{R(x)}$.} $\Sigma_{R(x)}$ and
transformations $T_x,\,T_x^{-1}$ of class $C^{3,\phi}$. At first,
we consider $x\in\Omega$ within a sufficiently small distance
$d>0$ of $\Gamma(\alpha)$, where $\Gamma(\alpha)$ denotes the
portion of $\partial\Omega=\partial\Omega(\alpha)$ which is
determined by the design variable $\alpha$, see Figure
\ref{Figure_hemisphere_transformation}.

Because of the definition of the basic design and of the
admissible shapes $\alpha\in U^{ad}$ there is a $C^4$-extension
$\alpha^{\rm ext}:\Omega_{2d}^{ext}\rightarrow\mathbb{R}$ of $\alpha$
with $\Omega_{2d}\subset\Omega_{2d}^{ext}$ which describes a
portion $\Gamma^{ext}(\alpha)$ of the boundary
$\partial\Omega(\alpha)$ beyond $\Gamma(\alpha)$ and where
$\Omega_{2d}^{ext}$ is the image of a $C^4$-diffeomorphism
$\tilde{T}_{2d}:B(0,R)\subset\mathbb{R}^2\rightarrow\Omega_{2d}^{ext}$
for some $R>0$. This extension is needed in order to
consider all $x\in\Omega(\alpha)$ within a sufficiently small
distance $d>0$ of $\Gamma(\alpha)$. Now, we are able to define a
hemisphere transformation with the required properties:
\begin{equation}\label{SO.4.1}
    T_x:\Sigma_{R}\rightarrow\mathbb{R}^2,\quad T_x(x_1,x_2,x_3)=\left(%
\begin{array}{c}
(\tilde{T}_{2d}(x_1,x_2))_1\\
(\tilde{T}_{2d}(x_1,x_2))_2\\
  \alpha^{\rm ext}((\tilde{T}_{2d}(x_1,x_2))_1,(\tilde{T}_{2d}(x_1,x_2))_2)-x_3 \\
\end{array}%
\right),
\end{equation}
see Figure \ref{Figure_hemisphere_transformation}, too.
\begin{figure}[t]
\centering
\scalebox{0.37}{\input{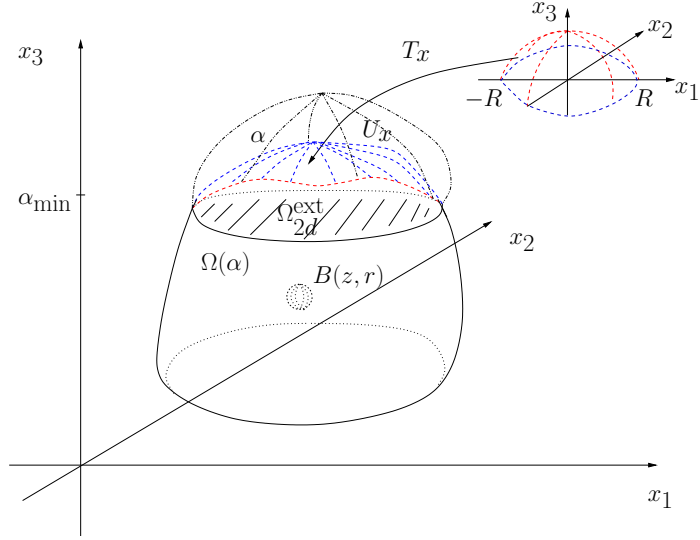}}
\caption{Hemisphere transformation.}
    \label{Figure_hemisphere_transformation}
\end{figure}
The neighborhood $U_x$ can be chosen so that
$\overline{U}_x\cap\overline{\Omega}=T_x(\Sigma_{R})$ and
$\overline{U}_x\cap\partial\Omega=T_x\left(F_{R}\right)$. This can
be achieved by sufficiently expanding $T_x(\Sigma_{R})$ beyond
$\alpha^{\rm ext}$. Because of the definition of the basic design
and of the design variables $\alpha$ one can find a bound for the
norms of the hemisphere transformations which is valid for all
$\alpha\in U^{ad}$. Analogously the remaining hemisphere
transformations can be constructed which are of a finite number
and all have a uniform bound denoted by $\kappa$, confer the last
part of the proof of Lemma \ref{lemma_compactness_graph}, too.
\qed
\end{proof}

The following lemma contains an inequality which can be found in
Section 7 of \cite{Agmon_Douglis_Nirenberg_1}. Using only a few
additional technical arguments, a statement about the inequality's
constant $C$ can be added regarding its dependency on cone
properties of the underlying domain $\Omega$.

\begin{lemma}\label{lemma_inequality}
Let $\mathcal{M}$ be a set of bounded domains in $\mathbb{R}^n$
with a uniform cone property and let $\Omega\in\mathcal{M}$. Then,
for every $\varepsilon>0$ there is a $C(\varepsilon)>0$ uniform
with respect to $\mathcal{M}$ such that
$\|v\|_{C^0(\Omega)}\leq\varepsilon\|v\|_{C^1(\Omega)}+C\int_\Omega|v|dx$
holds for all $v\in C^1(\Omega)$.
\end{lemma}
\begin{proof}
Without loss of generality let $v$ be non-negative. Furthermore,
assume that there is a $x_0$ with $v(x_0)=\|v\|_{C^0(\Omega)}$.
The case that there is no such $x_0$ can be treated similar.
Consider the cone with height $\|v\|_{C^0(\Omega)}$ and ground
area $B(x_0,R)$ with
\begin{align*}
R=\min\left(\frac{\|v\|_{C^0(\Omega)}}{(n\cdot\|Dv\|_{C^0(\Omega)})},C_1(\theta,h,r,n)\right)
\end{align*}
for some constant $C_1(\theta,h,r,n)$. Figure
\ref{Figure_inequality} shows for the one-dimensional case that
the ratio $\|v\|_{C^0(\Omega)}/\|Dv\|_{C^0(\Omega)}$ can be used
as radius for a cone which possesses the height
$\|v\|_{C^0(\Omega)}$ and a segment that is included under the
graph of $v$ if $\Omega=[a,b]$ is sufficiently large. Constant
$C_1(\theta,h,r,n)$ considers the case when $\Omega$ is not
sufficiently large. Thus, it follows from the cone property that
there is a constant $C_2(\theta,h,r,n)$ such that a segment of
the cone fits under the graph of $v$ 
and has the volume
\begin{equation*}
\min
\left(\frac{\|v\|_{C^0(\Omega)}^{n}}{\|Dv\|_{C^0(\Omega)}^{n}},1\right)C_2(\theta,h,r,n)\cdot\|v\|_{C^0(\Omega)}
\end{equation*}
which is a lower bound of the integral $\int_\Omega vdx$. 

\begin{figure}[htbp]
\centering
\scalebox{0.45}{\input{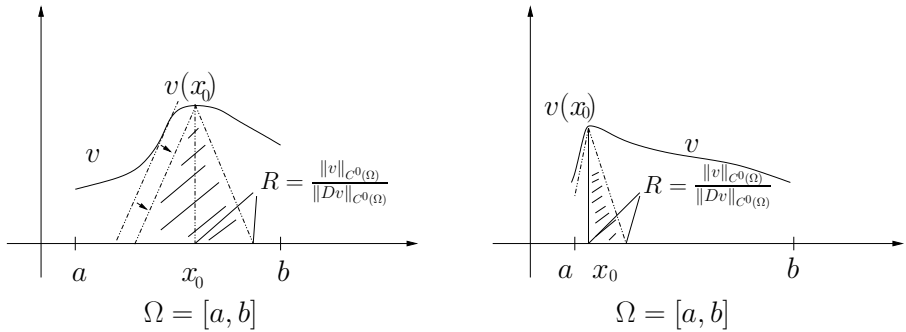}}
\caption{Cone segment under the graph of $v$: The corresponding
ground area is determined by the greatest increase of $v$ and the
cone property of $\Omega$. If $\Omega$ is not sufficiently large a
truncated cone segment has to be considered.}
    \label{Figure_inequality}
\end{figure}

If the minimum equals the first argument we show that for every
$\varepsilon>0$ there exists a $C(\varepsilon,\theta,h,r,n)>0$
such that
\begin{align*}
\|v\|_{C^0(\Omega)}\leq\varepsilon(\|v\|_{C^0(\Omega)}+\|Dv\|_{C^0(\Omega)})+C\frac{\|v\|_{C^0(\Omega)}^{n+1}}{\|Dv\|_{C^0(\Omega)}^{n}}
\end{align*}
which is equivalent to $1\leq\varepsilon(1+t)+Ct^{-n}$ for
$t=\|Dv\|_{C^0(\Omega)}/\|v\|_{C^0(\Omega)}$. The function on the
right side has its global minimum at
$t_0=\sqrt[n+1]{nC/\varepsilon}$ for positive values of $t$.
Because the minimum is
\begin{align*}
\varepsilon\left(1+\sqrt[n+1]{\frac{nC}{\varepsilon}}\right)+\left(\frac{n}{C^{1/n}\varepsilon}\right)^{-n/(n+1)}
\end{align*}
it is greater or equal than $1$ if $C\geq(n/\varepsilon)^n$.

If the minimum of the volume of the cone segment equals
$C_2(\theta,h,r,n)\|v\|_{C^0(\Omega)}$ we show that for every
$\varepsilon>0$ there exists a $C(\varepsilon,\theta,h,r,n)>0$
such that
\begin{align*}
\|v\|_{C^0(\Omega)}\leq\varepsilon(\|v\|_{C^0(\Omega)}+\|Dv\|_{C^0(\Omega)})+C\|v\|_{C^0(\Omega)}
\end{align*}
which is equivalent to $1\leq\varepsilon(1+t)+C$ for
$t=\|Dv\|_{C^0(\Omega)}/\|v\|_{C^0(\Omega)}$. So $C\geq1$ can
additionally be required and one finally obtains
$\max((n/\varepsilon)^n,1)\leq C(\varepsilon,\theta,h,r,n)$. \qed
\end{proof}

Before applying the Schauder estimates of Theorem 9.3 in
\cite{Agmon_Douglis_Nirenberg_2} to the solutions of
$\mathcal{P}(\alpha)$, we introduce the following technical
notations: For arbitrary multi-indices $\varrho$ and $\beta$
define $\delta_{\varrho\beta}$ to be one if $\varrho=\beta$ and to
be zero in any other case. A vector and a matrix of multi-indices
is given by $\varrho(i)$ and $\beta(ij)$ for $i=1,\dots,n$ and
$j=1,\dots,m$, respectively, where each $\varrho(i)$ and
$\beta(ij)$ represents a multi-index. The $k$-th component is
given by $\varrho(i)_k$ and $\beta(ij)_k$, respectively.

\begin{theorem}\label{theorem_Schauder_estimate_applied}
Let the state problem $\mathcal{P}(\alpha)$ be given on a domain
$\Omega(\alpha)\in\mathcal{O}$. Suppose that the Lam\'e
coefficients are constants. Moreover, let $f\in
[C^{1,\phi}(\overline{\Omega^{\textrm{ext}}})]^3$ and $g\in
[C^{2,\phi}(\overline{\Omega^{\rm ext}})]^3$ and $\phi\in(0,1)$.
Then, there are hemisphere transformations of class $C^{3,\phi}$
with uniform bound $\kappa$ and a unique solution $u\in
V(\Omega(\alpha))$ of $\mathcal{P}(\alpha)$ which also belongs to
$[C^{3,\varphi}(\overline{\Omega(\alpha)})]^3$ and satisfies
\begin{equation}\label{SO.3.2}
\|u\|_{[C^{3,\varphi}(\Omega)]^3}\leq
C\left(\|f\|_{[C^{1,\phi}(\Omega)]^3}+\|g\|_{[C^{2,\phi}(\partial\Omega)]^3}
+\|u\|_{[C^0(\Omega)]^3}\right)
\end{equation}
for any $\varphi\in(0,\phi)$ and some constant $C$. The term
$\|u\|_{[C^0(\Omega)]^3}$ can be replaced by $\int_\Omega|u|dx$
and $C$ can be chosen uniformly with respect to $\mathcal{O}$.
\end{theorem}
\begin{proof}
Let first $f\in [C^{2,\phi}(\overline{\Omega^{\rm ext}})]^3$ such
that its restriction to any $\Omega(\alpha)$ is in
$W^{2,p}(\Omega(\alpha))$ for $p>1$ arbitrary. This additional
condition will be eliminated in Lemma \ref{lemmaC2C1} below.

The existence of a unique solution $u\in V(\Omega(\alpha))$ is a
consequence of Theorem \ref{corollary_regularity_weak_solution}:
Because $\Omega(\alpha)$ has $C^4$-boundary we have a unique
solution $u\in[W^{4,p}(\Omega(\alpha))]^3$ for arbitrary
$p\geq6/5$. Then, the general Sobolev inequalities -- confer
Section 5.6 in \cite{Evans} -- will lead to $u\in
[C^{3,\phi}(\overline{\Omega(\alpha)})]^3$ 
if $p$ is sufficiently large\footnote{Note that $f\restriction_{\Omega(\alpha)}$ and $g\restriction_{\Omega(\alpha)}$,
 are continuously differentiable.}.

Now, we show that the assumptions for the Schauder estimates of
Theorem 9.3 in \cite{Agmon_Douglis_Nirenberg_2} are satisfied.
Main assumptions are complementing and supplementary boundary
conditions, uniform ellipticity with the corresponding constant
$A$, a positive minor constant $\Delta_{\partial\Omega}$ and the
existence of hemisphere transformations\footnote{Recall Definition
\ref{definition_hemisphere_property} where also constant $d$ is
defined.} of class $C^{3,\phi}$ with corresponding norms uniformly
bounded by a constant $\kappa$, confer Sections 1, 2, 7 and 9 of
\cite{Agmon_Douglis_Nirenberg_2} for a detailed description. We
start by rewriting the equations of $\mathcal{P}(\alpha)$ in the
form of $\sum_{j=1}^3l_{ij}(x,\nabla)u_j(x)=f_i(x)$ for
$x\in\Omega$, $i=1,2,3$ and of $\sum_{j=1}^3B_{hj}(x,\nabla)u_j(x)=g_h(x)$
for $x\in\partial\Omega_N$ and $h=1,2,3$, see (1.1) and (2.1) in
\cite{Agmon_Douglis_Nirenberg_2}. Therefore, we write
\begin{align}\label{Elasticity.5.2}
\begin{split}
l_{ij}(x,\Xi)&=\sum_{|\varrho|=0}^2a_{ij,\varrho}(x)\Xi^\varrho\\
&=\sum_{|\varrho|=2}\left[\mu(\delta_{ij}(\delta_{\varrho(2,0,0)}
+\delta_{\varrho(0,2,0)}+\delta_{\varrho(0,0,2)})+\delta_{\varrho\gamma(ij)})
+\lambda\sum_{\beta(i)_i>0}\delta_{\varrho\beta(i)}\right]\Xi^\varrho,
\end{split}
\end{align}
where
$\gamma(ij)=(\delta_{1i}+\delta_{1j},\delta_{2i}+\delta_{2j},\delta_{3i}+\delta_{3j})$
and where $\sum_{\beta(i)_i>0}$ is the sum over all multi-indices
with $\beta(i)_i>0$. Corresponding to the Neumann boundary
conditions on $\partial\Omega_N$ we write
\begin{align}\label{Elasticity.5.3}
\begin{split}
B_{hj}(x,\Xi)&=\sum_{|\varrho|=0}^1b_{hj,\varrho}\Xi^\varrho=\sum_{|\varrho|=1}
(\lambda\nu_h(x)\delta_{\varrho\beta(j)}+\delta_{hj}\mu
\nu_\varrho(x)+\mu\nu_j(x)\delta_{\varrho\beta(h)}) \Xi^\varrho,
\end{split}
\end{align}
where $\beta(1)=(1,0,0),\beta(2)=(0,1,0),\beta(3)=(0,0,1),
\nu_{(1,0,0)}=\nu_1,\nu_{(0,1,0)}=\nu_2,\nu_{(0,0,1)}=\nu_3$.
Regarding Theorem 9.3 in \cite{Agmon_Douglis_Nirenberg_2} and
equations (\ref{Elasticity.5.2}) and (\ref{Elasticity.5.3}) we
have $s_i=0,t_j=2$ and $r_h=-1$ for the Neumann condition and
$r_h=-2$ for the the Dirichlet condition\footnote{Note that
homogeneous Dirichlet conditions on $\partial\Omega_D$ lead to
$B_{hj}(x,\Xi)=\delta_{hj}$, $g_h(x)=0$ for all
$x\in\partial\Omega_D$ and to $r_h=-2$.} which implies
$l_0=\max\{0,r_h\}=0$ and allows to choose $1=l\geq l_0$ for
$i,j,h\in\{1,2,3\}$.

Consider that the complementing conditions and the supplementary
boundary conditions are satisfied, confer Section 6.3 of
\cite{Elasticity_Ciarlet_1}. The existence of appropriate
hemisphere transformations of class $C^{3,\phi}$ is shown in Lemma
\ref{lemma_hemisphere_trafo} where the constants $d$ and $\kappa$
can be chosen uniformly with respect to $\mathcal{O}$. According
to Section I.3 of \cite{Thompson} the minor constant
$\Delta_{\partial\Omega}$ is positive and determined by the Lam\'e
coefficients which are constant in our case.

Now, we analyze the effect of the hemisphere transformation $T_x$
on the ellipticity constant and follow Section 9 of
\cite{Agmon_Douglis_Nirenberg_2}. Let $c(\nabla)=\sum_\phi c_\phi
\nabla^\phi=\sum_{i_1,i_2,\dots}c_{i_1,i_2,\dots}\frac{\partial^{i_1}}{\partial
x_1^{i_1}}\frac{\partial^{i_2}}{\partial x_2^{i_2}}\dots$ be an
arbitrary linear combination of differentiation operators. It is
transformed into
$\hat{c}(\hat{\nabla})=\sum_{i_1,i_2,\dots}\hat{c}_{i_1,i_2,\dots}\frac{\partial^{i_1}}{\partial
\hat{x}_1^{i_1}}\frac{\partial^{i_2}}{\partial
\hat{x}_2^{i_2}}\dots$, where $\frac{\partial}{\partial
\hat{x}_j}=\sum_{i}\frac{\partial x_i}{\partial
\hat{x}_j}\frac{\partial}{\partial x_i}$ and each
$\hat{c}_{i_1,i_2,\dots}$ is a linear combination of the
$c_{i_1,i_2,\dots}$ with coefficients that are products of the
$\frac{\partial \hat{x}_j}{\partial x_i}$. Correspondingly, we
obtain $c(\xi)=\hat{c}(\hat{\xi})$ with $\hat{\xi}=\frac{\partial
x_i}{\partial \hat{x}_j}\xi_i$. According to Section 1 in
\cite{Agmon_Douglis_Nirenberg_2} uniform ellipticity is described
by the inequality $A^{-1}\|\Xi\|_2^{2m}\leq|L(x,\Xi)|\leq
A\|\Xi\|_2^{2m}$ for all $\Xi\in\mathbb{R}^{n+1}$ and all
$x\in\overline{\Omega}$, where $L(x,\Xi)$ is the characteristic
determinant of the PDE-system. The determinant is invariant under
the hemisphere transformation in the sense that
$\hat{L}(\hat{x},\hat{\xi})=L(x,\xi)$. As the first derivatives of
$T_x$ and $T_x^{-1}$ exist and as these maps are in $x$ uniformly
bounded with respect to the norm $\|\cdot\|_{C^{3,\phi}}$ there is
a constant $\omega$ such that
$\omega^{-1}\|\hat{\xi}\|_2\leq\|\xi\|_2\leq\omega\|\hat{\xi}\|_2$
for every $x\in\mathcal{A}$, $\xi\in\Sigma_{R(x)}$ and
$\hat{\xi}=T_x(\xi)$. Confer Chapter I.6.2 of
\cite{Elliptic_PDE_Gilbarg}, too. Finally, this results in the
uniform ellipticity of the transformed system with the new
ellipticity constant $A\omega^{2m}$.

Applying the Schauder estimates to both cases of Neumann boundary
and homogeneous Dirichlet conditions yields the inequality
statement, which even holds for\footnote{Recall that we will treat
the case $f\in [C^{1,\phi}(\overline{\Omega^{\textrm{ext}}})]^3$
in Lemma \ref{lemmaC2C1} where then $0<\varphi<\phi$ is required.}
$0<\varphi\leq\phi$. A uniform choice of the constant $C$ in
(\ref{SO.3.2}) with respect to $\mathcal{O}$ is justified by the
previous analysis of the constants
$d,\kappa,\Delta_{\partial\Omega}$ and $A$. The replacement of
$\|u\|_{[C^0(\Omega)]^3}$ by $\int_\Omega|u|dx$ is a consequence
of a uniform cone property of $\mathcal{O}$ and of Lemma
\ref{lemma_inequality}. \qed
\end{proof}

By means of the following theorem and the properties of
$C^k$-admissible domains we can show compactness of
$\mathcal{G}=\{(\Omega,u(\Omega))\,|\,\Omega\in\mathcal{O}\}$
where $u(\Omega)$ uniquely solves $\mathcal{P}(\Omega)$.

\begin{theorem}\label{corollary_boundedness}
There is a positive constant $C$ such that the solution $u\in
V(\Omega(\alpha))$ of the previous theorem satisfies
$\|u\|_{[C^{3,\varphi}(\Omega)]^3}\leq C$ for any
$\varphi\in(0,\phi)$, where $C$ can be chosen uniformly with
respect to $\mathcal{O}$ and forces $f\in
[C^{1,\phi}(\overline{\Omega^{{\rm ext}}})]^3,g\in
[C^{2,\phi}(\overline{\Omega^{\rm ext}})]^3$.
\end{theorem}
\begin{proof}
First note that the norms $\|f\|_{[C^{1,\phi}(\Omega)]^3}$ and
$\|g\|_{[C^{2,\phi}(\partial \Omega)]^3}$ in (\ref{SO.3.2}) are
uniformly (in $\Omega\in{\cal O}$) bounded by
$\|f\|_{[C^{1,\phi}(\overline{\Omega^{\rm ext}})]^3}$ and
$\|g\|_{[C^{2,\phi}(\overline{\Omega^{\rm ext}})]^3}$,
respectively.

The main part of the proof thus consists of showing that there is
a constant $C$ independent of $\Omega\in\mathcal{O}$ such that
$\|u\|_{H^1(\Omega(\alpha))}\leq C$. We follow ideas of the proof
of Lemma 2.24 in \cite{Shape_Optimization_Haslinger}:

The weak formulation (\ref{Elasticity.1.2}) of
$\mathcal{P}(\alpha)$ can be rewritten in the form
\begin{align*}
\int_\Omega{\rm tr}(\sigma(u)\varepsilon(v))\,dx
=\int_\Omega f\cdot v\,dx
+\int_{\partial\Omega_N}g\cdot v\,dA\quad\forall v\in
V_{\textrm{DN}},
\end{align*}
where $\sigma_{ij}=\sum_{k,l=1}^3C_{ijkl}\varepsilon_{kl}$. The
ellipticity constant
$C_{ijkl}=\delta_{ij}\delta_{kl}\lambda+\mu(\delta_{ik}\delta_{jl}+\delta_{il}\delta_{jk})$
is given by the Lam\'e coefficients $\lambda,\mu$ and so a constant
element of $C(\overline{\Omega^\textrm{ext}})$. Moreover, the
constant satisfies the symmetries $C_{ijkl}=C_{jikl}=C_{klij}$ and
there exists a constant $q>0$ such that
$C_{ijkl}(x)\xi_{ij}\xi_{kl}\geq q\xi_{ij}\xi_{kl}$ for all
$x\in\overline{\Omega^\textrm{ext}}$. This results in
\begin{align*}
B_\alpha(v,v)=\int_{\Omega(\alpha)}{\rm tr}(\sigma(v)\varepsilon(v))dx
\geq
q\int_{\Omega(\alpha)}{\rm tr}(\varepsilon(v)^2)\,dx
=q\|\varepsilon(v)\|_{[L^2(\Omega(\alpha))]^{3\times 3}}^2
\end{align*}
for all $v\in V_{\textrm{DN}}$. Because of the assumptions for $f$
and $g$
\begin{align*}
\left|L_\alpha(v)\right|=\left|\int_{\Omega(\alpha)} f\cdot v\,dx
+\int_{\partial\Omega_N(\alpha)} g\cdot v\,dA\right|\leq
C\|v\|_{[H^1(\Omega(\alpha))]^3}
\end{align*}
holds and the constant is uniform in $\alpha\in U^{\rm ad}$ as a
result of the uniform bounds on $\|f\|_{[C^{1,\phi}(\Omega)]^3}$
and $\|g\|_{[C^{2,\phi}(\partial \Omega)]^3}$.  This independence
is a consequence of
$\|v\restriction_{\partial\Omega}\|_{[L^2(\partial\Omega)]^3}\leq
C\|v\|_{[H^1(\Omega)]^3}$ with $C$ depending only on the uniform
Lipschitz constant of the boundary, confer \cite{Necas}. By the
same reason, the volume of $\partial\Omega$ is uniformly bounded.
The previous inequalities and the weak equation
(\ref{Elasticity.1.2}) lead to
$q\|\varepsilon(u)\|_{[H^0(\Omega(\alpha))]^3}^2\leq
C\|u\|_{[H^1(\Omega(\alpha)]^3}$. This and Korn's second
inequality (\ref{theorem_Korns_first_inequality}) imply
\begin{equation*}
q\|\varepsilon(u)\|_{[H^0(\Omega(\alpha))]^3}^2\leq
C\|u\|_{[H^1(\Omega(\alpha))]^3}\leq
C\|\varepsilon(u)\|_{[H^0(\Omega(\alpha))]^3},
\end{equation*}
where $C$ also depends on the constant of Korn's second
inequality. As \cite{Nitsche} shows for the local epigraph
parametrization, this constant is uniform with respect to a class
of domains possessing a uniform cone property. Applying once more
Korn's second inequality one obtains
$\|u\|_{[H^1(\Omega(\alpha))]^3}\leq C$ with $C$ independent of
$\alpha$.

This result, the inequality $\|v\|_{L^1(\Omega)}\leq
\sqrt{\textrm{vol}(\Omega)}\,\|v\|_{L^2(\Omega)}$ for $v\in
L^2(\Omega)$ and the results of the previous theorem show the
statement of this theorem. $\qed$
\end{proof}

The following lemma closes the gap left open in the proof of
Theorem \ref{theorem_Schauder_estimate_applied}, as so far we have
only proven these statements for $f\in
[C^{2,\phi}(\overline{\Omega^{\rm ext}})]^3$ in order to be able
to apply Theorem \ref{corollary_regularity_weak_solution} -- see
the beginning of the proof of Theorem
\ref{theorem_Schauder_estimate_applied}. Since the right hand side
of the Schauder estimate (\ref{SO.3.2}) only depends on the
$C^{1,\phi}$-norm of $f$, it is possible to overcome this
restriction:

\begin{lemma}\label{lemmaC2C1}
Suppose that the statements of Theorem
\ref{theorem_Schauder_estimate_applied} and Theorem
\ref{corollary_boundedness} hold for $f\in
[C^{2,\phi}(\overline{\Omega^{\rm ext}})]^3$. Then, they extend to
$f\in [C^{1,\phi}(\overline{\Omega^{\rm ext}})]^3$ with the same
uniform constant $C$.
\end{lemma}
\begin{proof}
Finally, we have to consider the case $f\in
[C^{1,\phi}(\overline{\Omega^{\rm ext}})]^3$. Then, there exists a
sequence $(f_n)_{n\in\mathbb{N}}\subset
[C^{2,\phi}(\overline{\Omega^{\rm ext}})]^3$ such that $f_n\to f$
in $ [C^{1,\phi}(\overline{\Omega^{\rm ext}})]^3$ as $n\to
\infty$. Given $\Omega(\alpha)$, denote the sequence of solutions
to ${\cal P}(\alpha)_n$ with volume force $f_n$ by
$u_n,\,n\in\mathbb{N}$. We can now apply Theorem
\ref{corollary_boundedness} to $u_n$ with $\varphi=\phi$ due to
$f_n\in [C^{2,\phi}(\overline{\Omega^{\rm ext}})]^3$, confer the
last paragraph in the proof of Theorem
\ref{theorem_Schauder_estimate_applied}. Thus,
$(u_n)_{n\mathbb{N}}$ is uniformly bounded in
$[C^{3,\phi}(\overline{\Omega^{\rm ext}})]^3$. According to
compact embeddings of H\"older spaces (a consequence of Arzela
Ascoli), there is a subsequence $(u_{n_k})_{n_k\in\mathbb{N}}$
such that $u_{n_k}\to u$ in $[C^{3,\varphi}(\overline{\Omega^{\rm
ext}})]^3$ for any $\varphi\in(0,\phi)$. We deduce that $u$
fulfills ${\cal P}(\alpha)$ due to pointwise convergence of first
and second derivatives. As the Schauder estimate (\ref{SO.3.2})
holds for all $u_{n_k}$, it carries over to $u$ by
$C^{3,\varphi}$-continuity in $u$ of both sides and the
$C^{1,\phi}$-continuity in $f$ of the right hand side of that
inequality. \qed
\end{proof}

\section{Existence of Optimal Shapes}\label{section_optimal_reliability}

In this final section, we exploit the results of the previous
section in order to prove existence of solutions to shape
optimization problems which are given by a very general class of
cost functionals. These cost functionals are not constrained by
convexity assumptions and the state problems are described by
mixed problems of linear elasticity, see Definition
\ref{state_problem}. We optimize shapes within the family of
$C^k$-admissible domains. The class of shape optimization problems
is large enough to include those originating from optimal
reliability of local crack initiation models up to third order,
confer Lemma \ref{opt_reliability_is_shape_opt} and Definition
\ref{stress_models}.

Since the cost functionals include surface integrals which lead to
a loss of regularity according to the trace
theorem\footnote{Confer Appendix B.3.5 in
\cite{Finite_Elements_Ern} for more details.} and (higher) derivatives of
$u$, we have to resort to strong regularity of solutions. As
already announced in Section \ref{section_SO_admissible_domains}
only a part of the boundary is subject of optimization. The
abstract setting of Section \ref{section_SO_admissible_domains}
and Theorem \ref{theorem_optimum_shape_design} determine the
structure of this section which leads to our existence results.

As our state space $V(\Omega(\alpha))$ is
$[C^{3}(\Omega(\alpha))]^3$, we employ the following definition
with $q=3$ for the convergence of functions with variable domains
in $\mathcal{O}$, also confer Section 2.5.2 in
\cite{Shape_Optimization_Haslinger}. Following the previous
section, we consider $C^4$-admissible shapes $\mathcal{O}$.

\begin{definition}\label{definition_Ck_function_convergence}\textbf{($C^{q}$-Convergence of Functions with Variable Domains)}\rm\\
Recalling the sets $\mathcal{O}$ and $\Omega^{\textrm{ext}}$ of
Definition \ref{definition_design_variables_admissible_domains}
let $p_\Omega:[C^{q}(\overline{\Omega})]^3\rightarrow
[C^{q}_0(\Omega^{\textrm{ext}})]^3$ be the extension operator
which can be derived from Lemma
\ref{appendix_function_spaces_theorem.5.1} for
$q\in\mathbb{N}\setminus\{0\}$. For $u\in
[C^{q}(\overline{\Omega})]^3$ set $u^{\textrm{ext}}=p_\Omega u$.
For $(\Omega_l)_{l\in\mathbb{N}}\subset\mathcal{O}$,
$\Omega\in\mathcal{O}$ and $(u_l)_{l\in\mathbb{N}}$ with $u_l\in
[C^{q}(\overline{\Omega}_l)]^3$, $l\in\mathbb{N}$, and
$u\in[C^{q}(\overline{\Omega})]^3$ the expression
$u_l\rightsquigarrow u$ as $l\rightarrow\infty$
is defined by $u^{\textrm{ext}}_l\rightarrow u^{\textrm{ext}}$ in
$[C^{q}_0(\Omega^{\textrm{ext}})]^3$.
\end{definition}

\noindent \textbf{Notation (Local Cost Functional)}\\
Find an optimal shape $\Omega(\alpha)\in\mathcal{O}$ which
minimizes a local functional of the form
$J(\Omega,u)=J_{\textrm{vol}}(\Omega,u)+J_{\textrm{sur}}(\Omega,u)$
with
\begin{align}\label{cost_functionals_so}
\begin{split}
J_{\textrm{vol}}(\Omega,u)&=\int_\Omega \mathcal{F}_{\rm
vol}(x,u,\nabla u,\nabla^2
u,\nabla^3u)\,dx,\\
 J_{\textrm{sur}}(\Omega,u)&=\int_{\partial\Omega}
\mathcal{F}_{\rm sur}(x,u,\nabla u,\nabla^2
u,\nabla^3u)\,dA
\end{split}
\end{align}
and $u$ uniquely\footnote{The uniqueness of $u$ is realized by
Theorem \ref{theorem_Schauder_estimate_applied}.} given by
$\Omega(\alpha)$ as the solution of the state problem
$\mathcal{P}(\alpha)$.

\begin{lemma}\label{lemma_compactness_graph}
Let the setting of Theorem \ref{corollary_boundedness} be given on
an arbitrary sequence of domains
$(\Omega(\alpha_n))_{n\in\mathbb{N}}\subset\mathcal{O}$. For
$\phi\in(0,1)$ let $f\in
[C^{1,\phi}(\overline{\Omega^{\textrm{ext}}})]^3$ and $g\in
[C^{2,\phi}(\overline{\Omega^{\textrm{ext}}})]^3$.  Let
$(\alpha_n,u_n)_{n\in\mathbb{N}}$ be a sequence of admissible
shapes $\alpha_n\in U^{\rm ad}$ and of the corresponding solutions
$u_n\in V(\Omega(\alpha_{n}))$ of $\mathcal{P}(\alpha_n)$. Then,
there is a subsequence $(\alpha_{n_k},u_{n_k})_{n_k\in\mathbb{N}}$
such that
$\Omega(\alpha_{n_k})\stackrel{\tilde{\mathcal{O}}}{\longrightarrow}\Omega(\alpha)$
and $u_{n_k}\rightsquigarrow u$ as $n_k\rightarrow\infty$ for some
$\alpha\in U^{\rm ad}$ and for the corresponding solution $u\in
V(\Omega(\alpha))$ of $\mathcal{P}(\alpha)$.
\end{lemma}
\begin{proof}
Due to Lemma \ref{lemma_compact_Ck_admissible_shapes} there is a
subsequence $(\alpha_{n_l})_{n_l\in\mathbb{N}}$ with
$\Omega(\alpha_{n_l})\stackrel{\tilde{\mathcal{O}}}{\longrightarrow}\Omega(\alpha)$
as $n_l\rightarrow\infty$ for some $\alpha\in U^{\rm ad}$.
According to Theorem \ref{theorem_Schauder_estimate_applied}
$u_{n_l}\in[C^{3,\varphi}(\overline{\Omega(\alpha_{n_l})})]^3$
holds for some $\varphi\in(0,\phi)$ and every $n_l\in\mathbb{N}$.
Moreover, according to Theorem \ref{corollary_boundedness} there
is a constant $C>0$ independent of every $\Omega\in\mathcal{O}$
such that
$\|u_{n_l}\|_{[C^{3,\varphi}(\Omega(\alpha_{n_l}))]^3}\leq C$.
Because of Lemma \ref{appendix_function_spaces_theorem.5.1} there
is a constant $C$ such that
$u^{\textrm{ext}}_{n_l}=p_{\Omega(\alpha_{n_l})}u_{n_l}$ satisfies
$\left\|u^{\textrm{ext}}_{n_l}\right\|_{[C^{3,\varphi}(\Omega^{\textrm{ext}})]^3}\leq
C\|u_{n_l}\|_{[C^{3,\varphi}(\Omega(\alpha_{n_l}))]^3}$.
The constant can be chosen to be independent of the design
variables $\alpha$ which we will show at the end of the proof.
The previous inequalities result in a uniform bound for all
$\left\|u^{\textrm{ext}}_{n_l}\right\|_{[C^{3,\varphi}(\Omega^{\textrm{ext}})]^3}$.
As the unit ball in $C^{3,\varphi}(\Omega^{\textrm{ext}})$ is
compact in $C^{3}(\Omega^{\textrm{ext}})$ due to the Arzela-Ascoli
theorem, we find a further subsequence
$(\alpha_{n_k},u_{n_k})_{n_k\in\mathbb{N}}$ such that
\begin{align*}
\Omega(\alpha_{n_k})\stackrel{\tilde{\mathcal{O}}}{\longrightarrow}\Omega(\alpha),\,
u^{\textrm{ext}}_{n_k}\stackrel{C^{3}}{\longrightarrow} v\textrm{
as }n_k\rightarrow\infty \textrm{ for some } v\in
[C^{3}(\Omega^\textrm{ext})]^3.
\end{align*}
Because of
$u^{\textrm{ext}}_{n_k}\stackrel{C^{3}}{\longrightarrow} v\textrm{
as }n_k\rightarrow\infty$ the function
$v|_{\overline{\Omega(\alpha)}}$ solves $\mathcal{P}(\alpha)$.
According to Theorem \ref{corollary_regularity_weak_solution}
$\mathcal{P}(\alpha)$ has a unique solution $u\in
V(\Omega(\alpha))$ and $v$ is an extension $u^\textrm{ext}$ of
$u$.

Finally, we show the statement that the constant $C$ of Lemma
\ref{appendix_function_spaces_theorem.5.1} can be chosen uniformly
with respect to the set of admissible shapes $\mathcal{O}$. The
proof of this extension lemma, confer the proof of Lemma 6.37 in
\cite{Elliptic_PDE_Gilbarg}, is based on
$C^{k,\varphi}$-diffeomorphisms $\psi$ that locally straighten the
boundary $\partial\Omega$. For $u\in
C^{k,\varphi}(\overline{\Omega})$ one considers
$\tilde{u}(y)=u\circ\psi^{-1}(y)$, where
$y=(y_1,\dots,y_{n-1},y_n)=(y',y_n)$ and $y_n>0$. An extension
into $y_n<0$ can then be defined by
\begin{equation*}
\tilde{u}(y',y_n)=\sum_{i=1}^{k+1}c_i\tilde{u}(y',-y_n/i),\quad
y_n<0,
\end{equation*}
where $c_1,\dots,c_{k+1}$ are determined by the equations
$\sum_{i=1}^{k+1}c_i(-1/i)^m=1$ with $m=0,\dots,k$.
The proof in \cite{Elliptic_PDE_Gilbarg} then shows that
$w=\tilde{u}\circ\psi$ provides a $C^{k,\varphi}$ extension of $u$
into $\Omega\cup B$ for some balls $B$ and corresponding
$C^{k,\varphi}$-diffeomorphisms $\psi$. Considering a
finite-covering argument of $\partial\Omega$ and a partition of
unity (subordinate to this covering) prove the existence of an
extension $w\in C^{k,\varphi}(\Omega^{\textrm{ext}})$. The
inequality
\begin{equation}\label{equ_extension_inequ}
\|w\|_{C^{k,\varphi}(\Omega^{\textrm{ext}})}\leq
C\|u\|_{C^{k,\varphi}(\Omega)}
\end{equation}
for a constant $C$ depending only on $k,\Omega$ and
$\Omega^{\textrm{ext}}$ is a consequence of
\begin{equation}\label{equ_lipschitz}
K^{-1}\|x-y\|\leq\|x'-y'\|\leq K\|x-y\|,\quad
x'=\psi(x),\,y'=\psi(y),
\end{equation}
confer proof of Lemma 6.37 and equations (6.29) and (6.30) in
\cite{Elliptic_PDE_Gilbarg}. As we are employing a hypograph
representation for the admissible shapes we can subdivide the
problem by first considering the part of the boundary that is
fixed and then the part that varies with the admissible design
variables $\alpha\in U^{\textrm{ad}}$. The fixed part can be
treated uniformly for all admissible shapes. Thus, there are
always the same $C^{k,\varphi}$-diffeomorphisms $\psi$ with
respect to the admissible shapes for that part which can be used
for the construction of the extension as described above. For the
varying part of the boundary we explicitly give the form of the
$C^{k,\varphi}$-diffeomorphisms depending on the admissible shape.
Following the proof of Lemma 6.4 we obtain
\begin{equation*}
\psi(x)=\left(%
\begin{array}{c}
  x_1 \\
  x_2 \\
  \alpha(x_1,x_2)-x_3 \\
\end{array}%
\right),\quad \psi^{-1}(x')=\left(%
\begin{array}{c}
  x'_1 \\
  x'_2 \\
  \alpha(x'_1,x'_2)-x'_3 \\
\end{array}%
\right).
\end{equation*}
Because the norms of the admissible design variables $\alpha\in
U^{\textrm{ad}}$ are uniformly bounded so are the norms of $\psi$
and $\psi^{-1}$ and thereby constant $K$ of (\ref{equ_lipschitz})
can be chosen uniformly with respect to the admissible shapes.
This finally shows that $C$ of (\ref{equ_extension_inequ}) can
also be chosen uniformly. $\qed$
\end{proof}

As already mentioned we consider functionals which are volume or
surface integrals with continuous integrands. In order to show
continuity of the functionals we apply Lebesgue's dominated
convergence theorem.

\begin{lemma}\label{lemma_continuous_functional} \textbf{(Continuity of Local Cost
Functionals)}\\
Let $\mathcal{F}_{\textrm{vol}},\mathcal{F}_{\textrm{sur}}\in
C^0(\mathbb{R}^d)$ (with $d$ as in Definition \ref{stress_models}
with $r=3$) and let the set $\mathcal{O}$ only consist of
$C^0$-admissible shapes. For $\Omega\in\mathcal{O}$ and $u\in
[C^{3}(\overline{\Omega})]^3$ consider the volume integral $J_{\rm
vol}(\Omega,u)$ and the surface integral $J_{\rm sur}(\Omega,u)$
of (\ref{cost_functionals_so}), respectively.

Let $(\Omega_{n})_{n\in\mathbb{N}}\subset\mathcal{O}$ with
$\Omega_{n}\stackrel{\tilde{\mathcal{O}}}{\longrightarrow}\Omega$
as $n\rightarrow\infty$ and let
$(u_n)_{n\in\mathbb{N}}\subset[C^{3}(\overline{\Omega}_n)]^3$ be a
sequence with $u_n\rightsquigarrow u$ as $n\rightarrow\infty$ for
some $u\in [C^{3}(\overline{\Omega})]^3$. Then,
\begin{itemize}
\item[(i)] $J_{\rm vol}(\Omega_{n},u_{n})\rightarrow
J_{\rm vol}(\Omega,u)$ as $n\rightarrow\infty$.
\item[(ii)] If the set $\mathcal{O}$ only consists of $C^1$-admissible shapes
one obtains $J_{\rm sur}(\Omega_{n},u_{n})\rightarrow
J_{\rm sur}(\Omega,u)$ as $n\rightarrow\infty$ as well.
\end{itemize}
\end{lemma}

\begin{proof} (i)
At first, we consider the volume integral. Using the
characteristic function one obtains
\begin{align*}
J_{\rm vol}(\Omega_n,u_n)&=\int_{\Omega^{\textrm{ext}}}\chi_{\Omega_n}\cdot
\mathcal{F}_{\textrm{vol}}(x,u^{\textrm{ext}}_n,\nabla
u^{\textrm{ext}}_n,\nabla^2 u^{\textrm{ext}}_n,\nabla^3 u^{\textrm{ext}}_n)\,dx.
\end{align*}
Because of $\mathcal{F}_{\textrm{vol}}\in C^0(\mathbb{R}^d)$ and
$u_n\rightsquigarrow u$ as $n\rightarrow\infty$ there is a
constant $C>0$ such that the inequality
$\left|\chi_{\Omega_n}(x)\cdot
\mathcal{F}_{\textrm{vol}}(x,u^{\textrm{ext}}_n(x),\nabla
u^{\textrm{ext}}_n(x),\nabla^2 u^{\textrm{ext}}_n(x))\right|\leq
C$ holds for all $n\in\mathbb{N}$ almost everywhere in
$\Omega^{\textrm{ext}}$. Moreover,
$\Omega_{n}\stackrel{\tilde{\mathcal{O}}}{\longrightarrow}\Omega$
and $u^{\textrm{ext}}_n\rightarrow u^{\textrm{ext}}$ in
$[C^{3}_0(\Omega^{\textrm{ext}})]^3$ and ${\cal F}_{\rm vol}\in
C^0(\R^d)$ ensure the existence of
\begin{align*}
&\lim_{n\rightarrow\infty}\left(\chi_{\Omega_{n}}(x)\cdot
\mathcal{F}_{\textrm{vol}}(x,u^{\textrm{ext}}_{n}(x),\nabla
u^{\textrm{ext}}_{n}(x)),\nabla^2
u^{\textrm{ext}}_{n}(x)),\nabla^3 u^{\textrm{ext}}_n(x)\right)\\
&=\chi_{\Omega}(x)\cdot
\mathcal{F}_{\textrm{vol}}(x,u^{\textrm{ext}}(x),\nabla
u^{\textrm{ext}}(x),\nabla^2 u^{\textrm{ext}}(x),\nabla^3 u^{\textrm{ext}}(x))
\end{align*}
for all $x\in\Omega^{\textrm{ext}}$. As the intrgrands are pointwise and uniformly in $\Omega^{\rm ext}$ bounded, Lebesgue's dominated
convergence theorem can now be applied to permute integral and
limit:
\begin{align*}
&\quad\lim_{n\rightarrow\infty}J_{\rm vol}(\Omega_{n},u_{n})=\lim_{n\rightarrow\infty}\int_{\Omega^{\textrm{ext}}}\chi_{\Omega_{n}}
\mathcal{F}_{\textrm{vol}}(x,u^{\textrm{ext}}_{n},\nabla u^{\textrm{ext}}_{n},\nabla^2 u^{\textrm{ext}}_{n},\nabla^3 u^{\textrm{ext}}_{n})\,dx\\
&=\int_{\Omega^{\textrm{ext}}}\,\lim_{n\rightarrow\infty}\left(\chi_{\Omega_{n}}\cdot
\mathcal{F}_{\textrm{vol}}(x,u^{\textrm{ext}}_{n},\nabla
u^{\textrm{ext}}_{n},\nabla^2
u^{\textrm{ext}}_{n},\nabla^3 u^{\textrm{ext}}_{n})\right)\,dx\\
&=\int_{\Omega}\mathcal{F}_{\textrm{vol}}(x,u,\nabla
u,\nabla^2 u,\nabla^3u)\,dx=J_{\rm vol}(\Omega,u).
\end{align*}

(ii) With respect to the surface integral similar arguments can be used
and so we only address technical steps which are special to
integrating over a surface. Because of their definition the
boundary of every $\Omega\in\mathcal{O}$ is a differentiable
submanifold and the surface integral
\begin{align*}
&J_{\rm sur}(\Omega_n,u_n)=\int_{\partial\Omega_n}
\mathcal{F}_{\textrm{sur}}(x,u_n,\nabla
u_n,\nabla^2 u_n,\nabla^3 u_n)\,dA
\end{align*}
is the sum of integrals over a finite number $|I|$ of map areas
$\{\mathcal{A}^i_n\}_{i\in I}$. This number is the same for every
$\Omega\in\mathcal{O}$ due to the definition of the basic design
$\hat{\Omega}$ and of the design variables $\alpha\in U^{\rm ad}$. The
integrals have the form
\begin{align*}
\int_{\mathcal{A}^i_n}\mathcal{F}_{\textrm{sur}}(\varphi_n(s),u_n(\varphi_n(s)),\nabla
u_n(\varphi_n(s)),\nabla^2
u_n(\varphi_n(s),\nabla^3
u_n(\varphi_n(s)))\sqrt{g^{\varphi_n}(s)}ds
\end{align*}
where $\varphi^i_n:\mathcal{A}^i_n\rightarrow\mathbb{R}^3$ is a
chart and $g^{\varphi_n}(s)$ the corresponding Gram determinant.
The map
$\varphi_n^{\alpha}:\overline{\Omega}_{2d}\rightarrow\mathbb{R}^3,\varphi_n^{\alpha}(s_1,s_2)=(s_1,s_2,\alpha_n(s_1,s_2))^T$,
is the chart for the portion of the boundary $\partial\Omega_n$
which is described by the design variable 
$\alpha_n\in C^1(\overline{\Omega}_{2d})$.
Because of $\mathcal{F}_{\textrm{sur}}\in C^0(\mathbb{R}^{d})$ and
$u_n\rightsquigarrow u$ as $n\rightarrow\infty$ and because of the
fact that $g^{\varphi^\alpha_n}$ is bounded we can apply
Lebesgue's dominated convergence theorem again to prove $J_{\rm
sur}(\Omega_n,u_n)\rightarrow J_{\rm sur}(\Omega,u)$ as
$n\rightarrow\infty$. \qed
\end{proof}

Now, we can prove the existence of optimal shapes and present our
main result.

\begin{theorem}\label{theorem_existence_optimal_shape}\textbf{(Existence of Optimal
Shapes for SO Problems in Linear Elasticity)}\\
Let $\mathcal{F}_{\rm vol},\mathcal{F}_{\rm sur}\in
C^0(\mathbb{R}^d)$ and recall $\tilde{\mathcal{O}},\,\mathcal{O}$
of the previous definitions\footnote{$\mathcal{O}$ consists of
$C^4$-admissible shapes.}. Consider the state problem
$\mathcal{P}(\alpha)$ with $f\in
[C^{1,\phi}(\overline{\Omega^{\textrm{ext}}})]^3$ and $g\in
[C^{2,\phi}(\overline{\Omega^{\textrm{ext}}})]^3$ for some
$\phi\in(0,1)$ and consider the cost functional
$J(\Omega,u)=J_{\rm vol}(\Omega,u)+J_{\rm sur}(\Omega,u)$ with
\begin{align*}
\begin{split}
J_{\textrm{vol}}(\Omega,u)&=\int_\Omega \mathcal{F}_{\rm
vol}(x,u,\nabla u,\nabla^2
u,\nabla^3u)\,dx,\\
 J_{\textrm{sur}}(\Omega,u)&=\int_{\partial\Omega}
\mathcal{F}_{\rm sur}(\cdot,u,\nabla u,\nabla^2
u,\nabla^3u)\,dA
\end{split}
\end{align*}
as defined in (\ref{cost_functionals_so}) and with $u=u(\Omega)$
the unique solution of $\mathcal{P}(\alpha)$. Then, there exists
an optimal shape $\Omega^*\in\mathcal{O}$ such that
$J(\Omega^*,u(\Omega^*))\leq J(\Omega,u(\Omega))$ for all
$\Omega\in\mathcal{O}$.
\end{theorem}
\begin{proof}
Let $(\alpha_n,u_n)_{n\in\mathbb{N}}$ be a minimizing sequence of
$\inf\{J(\Omega,u(\Omega)\,|\,\Omega\in\mathcal{O})\}$, where
$u_n=u(\Omega_n)=u(\Omega(\alpha_n))$ is the unique solution of
the state problem $\mathcal{P}(\alpha_n)$. Then, because of Lemma
\ref{lemma_compactness_graph} there exists a subsequence
$(\alpha_{n_k},u_{n_k})_{n_k\in\mathbb{N}}$ such that
$\Omega(\alpha_{n_k})\stackrel{\tilde{\mathcal{O}}}{\longrightarrow}\Omega(\alpha)$
and $u_{n_k}\rightsquigarrow u$ as $n_k\rightarrow\infty$ for some
$\alpha\in U^{\rm ad}$ and for the corresponding solution $u\in
V(\Omega(\alpha))$ of $\mathcal{P}(\alpha)$. Since all
prerequisites of Lemma \ref{lemma_continuous_functional} are
satisfied we obtain $J(\Omega_{n_k},u_{n_k})\rightarrow
J(\Omega,u)\quad\textrm{as }n_k\rightarrow\infty$. Because
$(\alpha_{n_k},u_{n_k})_{n_k\in\mathbb{N}}$ is also a minimizing
sequence of $\inf\{J(\Omega,u(\Omega)\,|\,\Omega\in\mathcal{O})\}$
the admissible shape $\Omega^*=\Omega(\alpha)$ is an optimal
shape.\qed
\end{proof}

Next, we apply the previous theorem to the cost functional from
Lemma \ref{opt_reliability_is_shape_opt} and we get
solutions to the optimal reliability problem as an easy corollary
to Theorem \ref{theorem_existence_optimal_shape}.

\begin{theorem}\label{theorem_optimal_reliable_design}\textbf{(Optimal Reliability)}\\
Let $\mathcal{O}$ be $C^4$-admissible domains as described in
Definition \ref{definition_Ck_admissible_domains} and let
$\mathscr{V}_{\rm vol}=C^{1,\phi}(\overline{\Omega^{\rm ext}})]^3$
and $\mathscr{V}_{\rm sur}=C^{2,\phi}(\overline{\Omega^{\rm
ext}})$. Let furthermore $\gamma$ be a local crack initiation
model as in Definition \ref{stress_models} with order $r\leq 3$
and $0$-regular intensity functions $\varrho_{\rm vol}$ and
$\varrho_{\rm sur}$. Let $T$ be failure time defined as the
formation of the first crack associated to $\gamma$. Let
furthermore $t^*\in\mathscr{T}$ be given.
\begin{itemize}
\item[(i)] Then, there exists at least one shape $\Omega^*\in{\cal O}$ that minimizes the probability of failure up to time $t^*$, confer Definition \ref{optimal_reliability_problem}.
\item[(ii)] In particular, this applies to the local Weibull model for LCF, confer Definition \ref{definition_local_probabilistic_LCF}.
\end{itemize}
\end{theorem}
\begin{proof}
(i) As all prerequisites of Theorem
\ref{theorem_existence_optimal_shape} are satisfied by Lemma
\ref{opt_reliability_is_shape_opt}  we directly obtain the
existence of an optimal shape $\Omega^*\in\mathcal{O}$ from
Theorem \ref{theorem_existence_optimal_shape}.

(ii) Confer Proposition \ref{properties_Weibull_model} and note
that for $\Omega\in{\cal O}$ the condition $\|\frac{1}{N_{\rm
det}(\nabla u(\Omega))}\|_{L^m(\partial \Omega)}<\infty$ is
fulfilled since $\frac{1}{N_{\rm det}(\nabla u(\Omega))}$ is
continuous and hence bounded on $\partial\Omega$. Now apply (i).
\qed
\end{proof}

The strong convergence properties underlying Theorem
\ref{theorem_existence_optimal_shape} can also be applied to cost
functionals that do not belong to the class
(\ref{cost_functionals_so}). As an example, one can also prove
existence of solutions to the deterministic optimal LCF lifing
problem.
\begin{theorem}\label{theorem_optimal_deterministic_design}\textbf{(Deterministic LCF and Shape Optimization)}\\
Under the setting of Theorem \ref{theorem_optimal_reliable_design}
there exists a solution $\Omega^*\in{\cal O}$ to the deterministic
optimal LCF lifing problem of Definition
\ref{definition_deterministic_LCF}.
\end{theorem}
\begin{proof} Define the cost functional
$J_{\rm
sur}(\Omega,u)=-\inf_{x\in\partial\Omega}N_{\textrm{det}}(\nabla
u(x))=-T_{{\rm det},\Omega}$
on ${\cal O}\times [C^{3,\phi}(\overline{\Omega^{\rm ext}})]^3$ that needs to be minimized.
The proof of Theorem \ref{theorem_existence_optimal_shape} yields
all arguments needed except for the continuity property of
$J_{\textrm{sur}}$. But this follows from uniform convergence
required in Definition \ref{definition_Ck_function_convergence}.
\qed
\end{proof}

We have proven the existence of optimal designs for a general
class of cost functionals without convexity constraints and
including higher order derivatives of the state solutions. The
technically relevant case of optimal reliability for probabilistic
models of LCF and its deterministic counterpart is included. In
\cite{Gottschalk_Schmitz}, we address sensitivity analysis and the
existence of shape gradients in the probabilistic framework.


\appendix
\section{Appendix}\label{appendix}

\begin{definition}\label{definition_hemisphere_property} \textbf{(Hemisphere Property, Hemisphere Transformation)} \cite{Agmon_Douglis_Nirenberg_2}\rm\\
Let $\Omega\subset\mathbb{R}^n$ be a domain, $\Gamma$ be a regular portion of
$\partial\Omega$ and let $\mathcal{A}\subset\Omega$ be a
subdomain such that $\partial\mathcal{A}\cap\partial\Omega$ is in
the interior of $\Gamma$ in the $n$-dimensional sense and let
$d>0$ be a positive constant. If every $x\in\mathcal{A}$ within a
distance $d$ of $\partial\Omega$ has a neighborhood $U_x$ with
$\overline{U}_x\cap\partial\Omega\subset\Gamma,\,B(x,d/2)\subset
U_x$ and
\begin{equation*}
\overline{U}_x\cap\overline{\Omega}=T_x(\Sigma_{R(x)}),\quad\overline{U}_x\cap\partial\Omega=T_x\left(F_{R(x)}\right),\quad
    0<R(x)\leq1
\end{equation*}
for some hemisphere $\Sigma_{R(x)}$ and functions $T_x,\,T_x^{-1}$
of some class $C^{k,\phi}$, $\mathcal{A}$ is said to satisfy a
hemisphere property. Moreover, the functions $T_x$ are called
hemisphere transformations.
\end{definition}

\begin{definition}\label{definition_Lipschitz_Cka_boundary_Gilbarg} \textbf{($C^{k,\phi}$-Boundary, Lipschitz Boundary)} \cite{Elliptic_PDE_Gilbarg}\rm\\
Let $\Omega\subset\mathbb{R}^n$ be a bounded domain, let
$k\in\mathbb{N}$ and $0\leq\phi\leq1$. The subset $\Omega$ has a
$C^{k,\phi}$-boundary if at each point $x_0\in\partial\Omega$
there exists $B=B(x_0,r)$ for some $r>0$ and an injective
$C^{k,\phi}$-map
$\psi:B\rightarrow D\subset\mathbb{R}^n$
such that
$\psi(B\cap\Omega)\subset\mathbb{R}_+^n=\{x\in\mathbb{R}^n\,|\,x_n\geq0\}$,
$\psi(B\cap\partial\Omega)\subset\partial\mathbb{R}_+^n$ and
$\psi^{-1}\in C^{k,\phi}(D)$. If the maps $\psi$ are only
Lipschitz continuous $\Omega$ has a Lipschitz boundary.
\end{definition}

\begin{remark}\label{remark_Cka_boundary} \cite{Elliptic_PDE_Gilbarg}\rm\\
Let $\Omega$ be a bounded domain in $\mathbb{R}^n$ and let
$k\in\mathbb{N}$ and $0\leq\phi\leq1$. If every
$x_0\in\partial\Omega$ has a neighbourhood in which the boundary
is locally described by a graph of a $C^{k,\phi}$-function of
$n-1$ of the coordinates $x_1,\dots,x_n$ the domain $\Omega$ has a
$C^{k,\phi}$-boundary. Note that the converse is true if $k\geq1$.
\end{remark}

\begin{lemma}\label{appendix_function_spaces_theorem.5.1}\textbf{(Extension Lemma)} \cite{Elliptic_PDE_Gilbarg}, Part I, Section 6\\
Let $\Omega^{\textrm{ext}}\subset\mathbb{R}^n$ be open and let
$\Omega$ be a $C^{k,\phi}$-domain with
$\overline{\Omega}\subset\Omega^{\textrm{ext}}$, $k\geq1$ and
$0\leq\phi<1$. For $\phi=0$ it is $C^{k,0}=C^{k}$. If $u\in
C^{k,\phi}(\overline{\Omega})$ there is a function $w\in
C^{k,\phi}_0(\Omega^{\textrm{ext}})$ such that $w=u$ in
$\overline{\Omega}$ and
$\|w\|_{C^{k,\phi}(\Omega^{\textrm{ext}})}\leq
C\|u\|_{C^{k,\phi}(\Omega)}$ for a constant $C$ depending only on
$k,\Omega$ and $\Omega^{\textrm{ext}}$.
\end{lemma}

{\small \noindent \textbf{Acknowledgements:} This work has been supported
by the German federal ministry of economic affairs BMWi via an AG
Turbo grant. We thank the gas turbine technology department of the
Siemens AG for stimulating discussions, and particularly Dr. Georg
Rollmann for his many helpful and constructive suggestions.
Prof. Dr. Rolf Krause (USI Lugano) and the
Institute of Computational Science (ICS Lugano) also supported
this work.  Many suggestions of the referees helped to improve this article.}


\addcontentsline{toc}{chapter}{Bibliography}
   \begin{thebibliography}{marke}
   \bibitem{Agmon_Douglis_Nirenberg_1} S. Agmon, A. Douglis and L. Nirenberg, \textit{Estimates Near the Boundary for Solutions of
    Elliptic Partial Differential Equations Satisfying General Boundary Conditions I}, Communications On Pure And Applied Mathematics, Vol. XII,
    pp. 623-727, 1959.
    \bibitem{Agmon_Douglis_Nirenberg_2} S. Agmon, A. Douglis and L. Nirenberg, \textit{Estimates Near the Boundary for Solutions of
    Elliptic Partial Differential Equations Satisfying General Boundary Conditions II}, Communications On Pure And Applied Mathematics, Vol. XVII,
    pp. 35-92, 1964.
    \bibitem{ABFJ} G.\ Allaire, E.\ Bonneter, G.\ Francfort and F.\ Jouve,  \textit{Shape Optimization by the Homogenization Method}, Numer. Math., 76, 1997, 27-68.
    \bibitem{AK1} G.\ Allaire and R.\ V.\ Kohn, \textit{Optimal design for minimum weight and compliance in plane stress using extremel microstructures}, Eur.\ J.\ Mech.\ A Solids {\bf 12} (1993), 839-878.
        \bibitem{AJ} G.\ Allaire and F.\ Jouve, \textit{Existence of Minimizers for Non-Quasiconvex Functionals Arising in Optimal Design}, Ann. I.H.P., Anal. Nonlin., 15, 3, 1998, 301- 339.
        \bibitem{AK1} G.\ Allaire and R.\ V.\ Kohn, \textit{Optimal design for minimum weight and compliance in plane stress using extremel microstructures}, Eur.\ J.\ Mech.\ A Solids {\bf 12} (1993), 839-878.
    \bibitem{Alt} H. Alt, \textit{Lineare Funktionalanalysis}, fifth edition, Springer, Berlin, 2006.
    \bibitem{Arora} J.S. Arora, \textit{Introduction to Optimum Design}, McGraw-Hill Book Company, New York, 1989.
    \bibitem{Harders_Roesler} M. B\"aker, H. Harders and J. R\"osler, \textit{Mechanical Behaviour of Engineering Materials: Metals, Ceramics,
Polymers, and Composites}, German edition published by Teubner
Verlag (Wiesbaden, 2006), Springer, Berlin Heidelberg 2007.
    \bibitem{BS} A.\ Borzi and V.\ Schulz, \textit{Computational Optimization of Systems governed by Partial Differential Equations}, SIAM series on computational engineering, SIAM 2012.
    \bibitem{Braess} D. Braess, \textit{Finite Elemente - Theorie, schnelle L\"oser und Anwendungen in der Elastizit\"atstheorie}, fourth edition, Springer, Berlin, 2007.
    \bibitem{Bucur_Buttazzo} D. Bukur and G. Buttazzo, \textit{Variational Methods in Shape Optimization Problems}, first edition, Birkh\"auser, Boston, 2005.
    \bibitem{Elasticity_Ciarlet_1} P. Ciarlet, \textit{Mathematical Elasticity - Volume I: Three-Dimensional Elasticity},
    Studies in Mathematics and its Applications, Vol. 20, North-Holland, Amsterdam, 1988.
    \bibitem{DZ} M. C. Delfour and J.-P. Zolesio, \textit{Shapes and geometries}, (2nd Ed), Advances in Design and Control, SIAM 2011.
    \bibitem{Douglis_Nirenberg} A. Douglis and L. Nirenberg, \textit{Interior Estimates for Elliptic Systems of Partial Differential Equations},
    Communications On Pure And Applied Mathematics, Vol. VIII, pp. 503-538, 1955.
    \bibitem{Finite_Elements_Ern} A. Ern and J.-L. Guermond, \textit{Theory and Practice of Finite Elements}, Springer, New York, 2004.
    \bibitem{Epp} K. Eppler, \textit{Efficient Shape Optimization Algorithms for Elliptic Boundary Value Problems}, Habilitationsschrift, Univ. Chemnitz (2007).
        \bibitem{EU} K.\ Eppler and A. Unger,  \textit{Boundary control of semilinear elliptic euquations -- existence of optimal solutions},
        Control and Cybernetics vol. 26 (1997) No. 2, 249--259.
    \bibitem{Escobar_Meeker} L. A. Escobar and W. Q. Meeker, \textit{Statistical Methods for Reliability Data}, Wiley-Interscience Publication, New York, 1998.
    \bibitem{Evans} L. C. Evans, \textit{Partial Differential Equations}, second edition, AMS, Providence, 2010.
    \bibitem{Fedelich} B. Fedelich, \textit{A stochastic theory for the problem of multiple surface crack coalescence},
                    International Journal of Fracture 91 (1998) 2345.
    \bibitem{Fujii} N. Fujii, \textit{Lower Semicontinuity in Domain Optimization Problems}, Journal of Optimization Theory and Applications, Vol. 59, No. 3, 1988.
    \bibitem{Elliptic_PDE_Gilbarg} D. Gilbarg and N. S. Trudinger, \textit{Elliptic Partial Differential Equations of Second Order}, Springer, Berlin, 1998.
    \bibitem{Gottschalk_Schmitz} H. Gottschalk, R. Krause and S. Schmitz, \textit{Optimal Reliability in Design for Fatigue Life, Part II -- Shape Optimization and Sensitivity Analysis}, in preparation.
    \bibitem{Shape_Optimization_Haslinger} J. Haslinger and R. A. E. M\"akinen, \textit{Introduction to Shape Optimization - Theory, Approximation and Computation}, SIAM - Advances in Design and Control, 2003.
    \bibitem{Haug_Arora} E.J. Haug and J.S. Arora, \textit{Applied Optimal Design}, John Wiley \& Sons, New York, 1979.
    \bibitem{Haug_Choi_Komkov} E.J. Haug, K.K. Choi and V. Komkov, \textit{Design Sensitivity Analysis of Structural Systems}, Academic Press, Orlando, 1986.
    \bibitem{Hetnarski} R. B. Hetnarski and M. Reza Eslami, \textit{Thermal Stresses - Advanced Theory and Applications}, Solid Mechanics and Its Applications, Vol. 158, Springer, Berlin, 2009.
     \bibitem{Hoffmann} M. Hoffmann and T. Seeger, \textit{A Generalized Method for Estimating Elastic-Plastic Notch Stresses and Strains, Part 1: Theory}, Journal of Engineering Materials and Technology, 107, pp. 250-254, 1985.
        \bibitem{Kallenberg} O. Kallenberg, \textit{Random Measures}, Akademie-Verlag, Berlin 1983.
    \bibitem{Knop_Jones_Molent_Wang} M. Knop, R. Jones, L. Molent and L. Wang,
            \textit{On Glinka and Neuber methods for calculating notch tip strains under cyclic load spectra},
            International Journal of Fatigue, Vol. 22, (2000) 743--755.
    \bibitem{LNT} W. B. Liu, P. Neittaanm\"aki and D. Tiba, \textit{Existence for shape optimization problems in arbitrary dimension}, SIAM J.\ Contr.\ Optimization {\bf 41} (2003) 1440-1454.
    \bibitem{MVH} N. Mal\'esys, L. Vincent and F. Hild, \textit{A probabilistic model to predict the formation
and propagation of crack networks in thermal
fatigue}, International Journal of Fatigue {\bf 31}, 3 (2009) 565-574.
    \bibitem{Necas} J. Necas, \textit{Les methodes directes en theorie des equations elliptiques}, first edition, Academia, Prag, 1967.
    \bibitem{Neuber} H. Neuber, \textit{Theory of Stress Concentration for Shear-Strained Prismatical Bodies with Arbitrary Nonlinear Stress-Strain Law}, J. Appl. Mech. 26, 544, 1961.
    \bibitem{Nitsche} J. A. Nitsche, \textit{On Korn's second inequality}, RAIRO Anal. Numer. 15, pp. 237-248, 1981.
    \bibitem{Vormwald} D. Radaj and M. Vormwald, \textit{Erm\"udungsfestigkeit}, third edition, Springer, Berlin Heidelberg, 2007.
    \bibitem{Ramberg} W. Ramberg and W. R. Osgood, \textit{Description of Stress-Strain Curves by Three Parameters}, Technical Notes - National Advisory Committee For Aeronautics, No. 902, Washington DC., 1943
    \bibitem{SKRG1}  S. Schmitz, G. Rollmann, H. Gottschalk  and R.
Krause, \textit{Risk estimation for LCF crack initiation}, Proc.
ASME Turbo Expo 2013, GT2013-94899, arXiv:1302.2909v1.
    \bibitem{SKRG2} S. Schmitz,  G. Rollmann, H. Gottschalk
and R.\ Krause, \textit{Probabilistic analysis of the LCF crack initiation life for a turbine blade under
thermo-mechanical loading}, to appear Proc.\ Int.\ Conf.\ LCF 7 (September 13).
        \bibitem{SSBKRG} S. Schmitz, T. Seibel, T. Beck, G. Rollmann, R. Krause and H. Gottschalk, \textit{A probabilistic Model for LCF}, Computational Materials Science {\bf 79} (2013), 584-590.
        \bibitem{Schott} G. Schott, \textit{Werkstofferm\"udung - Erm\"udungsfestigkeit}, Deutscher Verlag f\"ur Grundstoffindustrie, forth edition, Stuttgart, 1997.
        \bibitem{Sokolowski_Zolesio} J. Sokolowski and J.-P. Zolesio, \textit{Introduction to Shape Optimization -  Shape Sensivity Analysis}, first edition, Springer, Berlin Heidelberg, 1992.
        \bibitem{Sornette} D. Sornette, T. Magnin and Y. Brechet, \textit{The Physical Origin of the Coffin-Manson Law in Low-Cycle Fatigue}, Europhys. Lett., 20 (5),
    pp. 433-438, 1992.
     \bibitem{Thompson} J. L. Thompson, \textit{Some Existence Theorems for the Traction Boundary Value Problem of
Linearized Elastostatics}, Arch. Rational Mech. Anal., 32, pp.
369-399, 1969.
  \end{thebibliography}
\end{document}